\documentclass{article}
\UseRawInputEncoding 
\usepackage{amsmath,amsfonts,amssymb,amsthm,enumerate,inputenc,mathtools}
\usepackage[french, british]{babel}
\usepackage{color}
\usepackage{hyperref}

\usepackage{dsfont}

\providecommand{\R}{\mathbb{R}}

\providecommand{\N}{\mathbb{N}}

\providecommand{\eps}{\varepsilon}

\def\longrightharpoonup{\relbar\joinrel\rightharpoonup}
\def\cv{\stackrel{w}{\longrightharpoonup}}
\def\cvwstar{\stackrel{w*}{\longrightharpoonup}}

\providecommand{\calC}{\mathcal{C}}
\newcommand{\calF}{\mathcal{F}}
\newcommand{\calS}{\mathcal{S}}

\newcommand{\calI}{\mathcal{I}}
\newcommand{\bbR}{\mathbb{R}}
\newcommand{\pS}{\partial\mathcal{S}}

\renewcommand{\leq}{\leqslant}
\renewcommand{\geq}{\geqslant}

\DeclarePairedDelimiter\abs{\lvert}{\rvert}%

\renewcommand{\div}{\operatorname{div}}
\newcommand{\curl}{\operatorname{curl}}

\newtheorem{Theorem}{Theorem}
\newtheorem{Definition}{Definition}
\newtheorem{Corollary}{Corollary}
\newtheorem{Proposition}{Proposition}
\newtheorem{Lemma}{Lemma}
\newtheorem{Remark}{Remark}

\begin{document}

\date{\today}
\title{Existence of weak solutions to the two-dimensional incompressible Euler equations in the presence of sources and sinks}
\author{Marco Bravin\footnote{BCAM - Basque Center for Applied Mathematics, Mazarredo 14, E48009 Bilbao, Basque Country - Spain. }, 
Franck Sueur \footnote{ Institut de Math\'ematiques de Bordeaux, UMR CNRS 5251,
Universit\'e de Bordeaux, 351 cours
de la Lib\'eration, F33405 Talence Cedex, France   $\&$ Institut  Universitaire de France} }

\maketitle

\begin{abstract}
A classical model for sources and sinks in a  two-dimensional perfect incompressible fluid occupying a bounded domain dates back to Yudovich's paper \cite{Yudo} in 1966.
 In this model, on the one hand, the normal component of the fluid velocity is prescribed on the boundary and is nonzero on an open subset of the boundary, corresponding either to sources (where the flow is incoming) or to sinks (where the flow is outgoing).
On the other hand the vorticity of the fluid which is entering into the domain from the sources is prescribed. 

In this paper we investigate the existence of weak solutions to this system by relying on \textit{a priori} bounds of the vorticity, which satisfies a transport equation associated with the fluid velocity vector field. 
Our results cover the case where the vorticity has a $L^p$  integrability in space, with $p $ in $[1,+\infty]$, and prove the existence of solutions obtained by compactness methods from viscous approximations. 
More precisely we prove the existence of solutions which satisfy the vorticity equation in the distributional sense in the case where $p >\frac43$,  in the renormalized sense  in the case where $p >1$, and in a  symmetrized sense  in the case where $p =1$. 
\end{abstract}

\newpage
\tableofcontents
\newpage

%%%%%%%%%%%%%%%%%%%%
\section{Introduction}

This paper focuses on the mathematical analysis of  a 2D perfect incompressible fluid occupying a bounded domain in the presence of sources and sinks. 
A classical model  dates back to Yudovich's paper \cite{Yudo} in 1966 where, on the one hand, the normal component of the fluid velocity is prescribed on the whole boundary of the fluid domain, and on the other hand the vorticity  is prescribed on the part of the boundary where the fluid is entering into the domain. 
The connected components of this part are called the sources whereas the connected components of the part of the boundary where the flow is exiting of the domain  are called the sinks.

More precisely let $ \Omega $ an open bounded connected simply-connected non-empty subset of $ \mathbb{R}^2 $ with smooth boundary. 
Let $N \geq 2$, $ 1 \leq n \leq N-1 $,  
 $$ \calI^{+} = \{1,\dots, n\}   , \quad \calI^{-} = \{n+1, \dots, N \}   \quad   \text{ and } \quad   \calI = \calI^{+}\cup \calI^{-}  .$$
For $i \in \calI$,  let $ \calS^{i} $ an open connected simply-connected non-empty subset of $ \mathbb{R}^2 $ compactly contained in $ \Omega $ with smooth boundary. We assume that the closures of the sets $ \calS^{i} $  are pairwise disjoint. 
The domain occupied by the fluid is 
$$ \calF = \Omega \setminus \overline{\bigcup_{i \in \calI} \calS^{i}} ,$$  and we split the boundary of the fluid domain into two parts: 
$$ \partial \calF^{+} = \bigcup_{i \in \calI^{+} } \pS^{i} \quad   \text{ and } \quad \partial \calF^{-} = \bigcup_{i \in \calI^{-} } \pS^{i} ,$$
respectively called outlet and inlet. 
An example of fluid domain is shown in Figure \ref{fig}.
Finally note that it is possible to deal with the case where the fluid is allowed to enter or exit through the exterior domain $ \partial \Omega $ but in this work we assume, for sake of simplicity,  that the boundary $\partial \Omega$ of $ \Omega $ is impermeable. 

\begin{figure}
\centering 
\includegraphics[scale= 0.65]{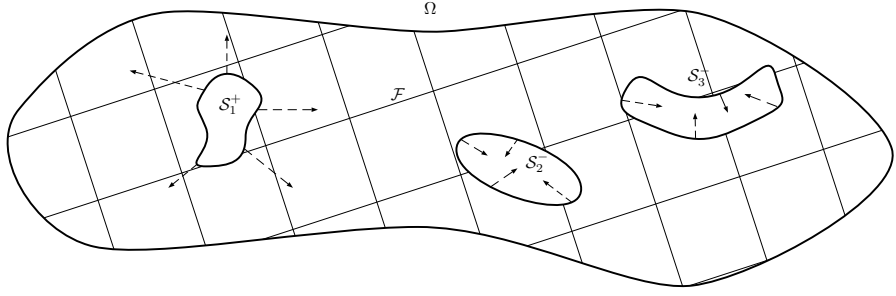}
\caption{Example of a fluid domain with one source and two sinks.}
\label{fig}
\end{figure}

The equations in the unknown $ (v,p) $ that model the dynamics read as 
\begin{subequations} \label{fir:sy:ss}
\begin{align}
\label{fir:sy:ss-eq} \partial_t v + v \cdot \nabla v +\nabla p = \, & 0 & & \text{ in } \mathbb{R}^{+} \times \mathcal{F},   \\ 
\label{fir:sy:ss-div}  \text{div} \, v = \, & 0 & & \text{ in } \mathbb{R}^{+} \times \mathcal{F},  \\
\label{fir:sy:ss-Y1} v\cdot n = \, & g & & \text{ on } \mathbb{R}^{+} \times \partial \mathcal{F},  \\
\label{fir:sy:ss-Y2} \curl v = \, & \omega^+ & & \text{ on } \bbR^{+} \times \partial \calF^{+},  \\
v(0,.) = \, & v^{in} & & \text{ in } \mathcal{F},   
\end{align}    
\end{subequations}
where $ v:\bbR^{+} \times \calF \to \mathbb{R}^2 $ is the fluid velocity field and $ p:\bbR^{+} \times \calF \to \bbR $ is the fluid pressure.
 In \eqref{fir:sy:ss-Y1},  $ n $ is the unit normal vector field exiting from the domain $ \mathcal{F} $. 
 The data $ g $ for the normal component of the velocity on the boundary is assumed to satisfy 
 $ g  < 0  $ on $ \bbR^{+} \times \pS^{i} $ with $ i \in \calI^{+} $, $ g  > 0 $ on $ \bbR^{+} \times \pS^{i} $ with $i \in \calI^{-}$, and $ g = 0 $ on $\bbR^{+}\times \partial \Omega $. 
 Because of these sign conditions we say that $ \calS^{i} $ is a source if $ i \in \calI^{+} $ and a sink if $i \in \calI^{-}$. The condition that  $ g = 0 $ on $\bbR^{+}\times \partial \Omega $ encodes that the external boundary $v$ is impermeable. 
 We also assume that, at any time $t$, the function $g(t,\cdot)$  has zero average on $\partial \calF$, which is the compatibility condition associated with the incompressibility, obtained from the integration  of the divergence free condition \eqref{fir:sy:ss-div} over  the whole fluid domain ${\calF}$.
 In \eqref{fir:sy:ss-Y2}, 
  $$ \omega^+: \bbR^+ \times \partial \calF^{+} \to \bbR ,$$
   is the entering vorticity. 
  Finally the initial data  $ v^{in} $ for the fluid velocity is assumed to satisfy $ \div v^{in} = 0 $ in $ \mathcal{F}$.

\begin{Definition}[Source-sink compatible] 
We say that a vector field $v: \bbR^+ \times \calF \to \bbR^2 $, respectively a function $ g: \bbR^+ \times \partial \calF \to \bbR $, is source-sink compatible (SSC)  if 
\begin{equation}
\label{SSC:v}
\div(v) = 0 \text{ in } \calF, \quad \int_{\partial \calF} v\cdot  n= 0, \quad v\cdot n = 0 \text{ on } \partial \Omega, \quad v\cdot n < 0 \text{ on } \partial \calF_{+}, \quad v \cdot n > 0 \text{ on } \partial \calF_{-}, \tag{SSC\textsubscript{1}}
\end{equation}
and respectively if
\begin{equation}
\label{SSC:g}
\int_{\partial \calF} g =  0, \quad g = 0 \text{ on } \partial \Omega, \quad g < 0 \text{ on } \partial \calF_{+}, \quad g > 0 \text{ on } \partial \calF_{-}. \tag{SSC\textsubscript{2}}
\end{equation}

Moreover, for a \ref{SSC:v} vector field $ v$,  we usually denote the normal trace on the boundary by $ g =  v \cdot n  $ on $ \partial \calF $.

\end{Definition}

\begin{Remark}
The choice of completing the system \eqref{fir:sy:ss} by prescribing the entering vorticity is not the only one. Another possibility is to prescribe a condition on the pressure, see  \cite{Mam}.
\end{Remark}

\begin{Remark}
Systems such as 
\eqref{fir:sy:ss} were extensively used in control theory,  see for instance  \cite[Section 6.2]{Coron},  \cite{G}, \cite{GR},   \cite{GKS}. 
Let us highlight that the sign conventions may differ here from some of these papers. 
\end{Remark}

%%%%%%%%%%%%%%%%%%%%
\subsection{Transport of  vorticity}
\label{sec-compo}

The velocity formulation \eqref{fir:sy:ss} is not well-adapted to a weak formulation and to energy estimates because of the pressure term,  in particular of its trace on the permeable part of the boundary. On the other hand Yudovich's boundary conditions  \eqref{fir:sy:ss-Y1}-\eqref{fir:sy:ss-Y2}  are very well adapted to a formulation in terms of the vorticity. In this paragraph we formally derive such a formulation by performing some computations from the velocity formulation above, assuming that we handle a smooth solution. First we apply the  $ \curl $  operator to the first equation of \eqref{fir:sy:ss} to obtain the following transport equation for the vorticity $ \omega $:
\begin{subequations} \label{tr-vort}
\begin{align}
\partial_t \omega + v \cdot \nabla \omega = & \, 0 && \text{ in } \bbR^{+} \times \calF, \label{corona}  \\
\omega = & \, \omega^+ && \text{ on } \bbR^{+} \times \partial \calF^{+}, \label{equ:vor:ss} \\
\omega(0,.) = & \, \omega^{in} && \text{ in } \calF .
\end{align}  
\end{subequations}

Formally one may solve this transport equation by the methods of characteristics. Because some fluid is entering through a part of the boundary one needs to distinguish two kinds of characteristics.
First for any $x\in  \calF$, 
we consider the position $X^i (t,x)$  at time $t>0$ of the fluid particle which is at the position $x$ at time $t=0$ and which moves following the velocity field $v$, that is we consider the ODE
\begin{equation}
	\begin{cases}
		\partial_t X^i (t,x)=v(t,X^i (t,x)), & t>0 \\[2mm]
		X^i (0,x)=x.
	\end{cases}
	\label{eq:char1}
\end{equation}
On the other hand, 
for any $z \in \partial  \calF_+$ and $s\geq0$, we consider 
the position $X^b (t,z,s)$   at time $t>0$ of the fluid particle which is at the position $z$ at time $t=s$
 and which moves following the velocity field $v$,  that is we consider the ODE
 \begin{equation}
	\begin{cases}
		\dfrac{d}{dt}X^b (t,z,s)=v(t,X^b (t,z,s)), & t>s \\[2mm]
		X^b (s,z,s)=z.
	\end{cases}
	\label{eq:char2}
\end{equation}
Then the vorticity at time $t$ can be formally recovered from  the initial vorticity and from the entering vorticity as follows. 
On   $[0,T]\times \calF$, 
we define: 
\begin{itemize}
\item $  \omega^i  $ by setting, for $(t,x) $ in $[0,T]\times  \calF$, 
 $ \omega^i (t,x)  =  \omega_0(y)$ when there is $y \in  \calF $ such that $x = X^i (t,y)$ and $ \omega^i (t,x)  =0$ otherwise, 
\item  $ \omega^b $  by setting, for $(t,x) $ in $[0,T]\times  \calF$, 
 $ \omega^b (t,x)  =  \omega_+ (t,z)$ when there is 
$(s,z)$ in  $[0,T]\times \partial { \calF}_+$ such that $x= X^b (t,z,s)$ 
and $ \omega^b (t,x)  =0$ otherwise. 
\end{itemize}

Then  for any  $(t,x) $ in $[0,T]\times  \calF$, the vorticity can be obtained as 
$$ \omega(t,x)= \omega^i (t,x) +  \omega^b (t,x) .$$

  Of course, this discussion is very formal since we did not not care about the Cauchy problem for \eqref{eq:char1} and \eqref{eq:char2}; indeed even in a smooth setting the flow has to be stopped or extended when the characteristics cross the outlet $\partial  \calF_+$.
 However this formal approach gives some insights on the  \textit{a priori}   bounds which may be true for the vorticity. 
 In particular it sounds reasonable to expect the following $L^p$  \textit{a priori}   bounds on the vorticity:  
 for $ 1 \leq q  < \infty$, 
\begin{equation} \label{apq}
\| \omega(t)\|^{q}_{L^{q}(\calF)} + \int_{0}^{t}\int_{\partial \calF^{-}} g |\omega^{-}|^{q} ds dt  \leq \| \omega^{in} \|_{L^{q}(\calF)}^{q}
+ \int_{0}^{t}\int_{\partial \calF^{+}} (-g) |\omega^{+}|^{q}ds dt ,  
\end{equation}
and 
\begin{gather}  \label{apinf}
  \max\left\{\| \omega\|_{L^{\infty}(0,t; L^{\infty}(\partial \calF^{+}))},   \| \omega^{-} \|_{L^{\infty}((0,t)\times \partial \calF^{-})}\right\}
\leq
  \max\left\{\| \omega^{in}\|_{L^{\infty}(\calF)}, \| \omega^i\|_{L^{\infty}(0,t; L^{\infty}(\partial \calF^{+}))} \right\} .
  \end{gather} 

In view of these   \textit{a priori}  bounds, we introduce the following definition regarding the integrability of 
the data of the problem concerning the vorticity, that is the initial vorticity $ \omega^{in}$ and the vorticity $\omega^{+}$ entering though the inlet $\partial \calF^{+}$.
\begin{Definition}[Couple of input vorticities]
\label{def-input}
For $p$ in $[1,+\infty]$, 
 we say that   $(\omega^{in} , \omega^{+}  )$ is a couple of input vorticities  in $L^p$ or shortly (CIV) in $L^p$ if 
\begin{equation}
\label{CIV}
\omega^{in} \in L^p(\calF, dx) \quad    \text{ and } \quad  \omega^{+} \in L^{p}_{loc}\left(\bbR^{+};L^{p}(\partial \calF^{+}, g ds)\right). \tag{CIV}
\end{equation}

\end{Definition}
Similarly, we introduce the following definition regarding the  integrability of the vorticity  $ \omega$  in the domain $\calF$ and 
the one of the 
the exiting vorticity that is the trace $\omega^{-} $ of $ \omega$ on the outlet $\calF_-$. 
\begin{Definition}[Continuous and weakly continuous couple]
\label{def-output}
For $p$ in $[1,+\infty ]$, 
 we say that   $(\omega , \omega^{-}  )$ is a continuous couple with values in $ L^p$
%in $L^p$ 
if 
\begin{equation}
\label{CC}
\begin{gathered} 
 \omega \in   C\left(\bbR^{+};  L^{p}(\calF)\right) \quad    \text{ and } \quad  \omega^{-} \in L^{p}_{loc}\left(\bbR^{+};L^{p}(\partial \calF^{-}, g ds)\right), 
\quad  \text{ if }  p < + \infty, 
\\ \omega \in   C\left(\bbR^{+};  L^{\infty}(\calF) -w*\right) \quad    \text{ and } \quad  \omega^{-} \in L^{\infty}_{loc}\left(\bbR^{+};L^{\infty}(\partial \calF^{-}, g ds)\right) , 
\quad  \text{ if }  p = + \infty . \tag{CC}
   \end{gathered} 
\end{equation}
Moreover for $ p \in[1, \infty)$  we say that   $(\omega , \omega^{-}  )$ is a weakly continuous couple with values in $ L^p$
%in $L^p$ 
if 
\begin{equation}
\label{WCC}
 \omega \in   C\left(\bbR^{+};  L^{p}(\calF)-w \right) \quad    \text{ and } \quad  \omega^{-} \in L^{p}_{loc}\left(\bbR^{+};L^{p}(\partial \calF^{-}, g ds)\right). \tag{WCC}
\end{equation}

\end{Definition}
Above the notation $L^{\infty}(\calF) -w*$ refers to the weak star topology of  $L^{\infty}(\calF)$ viewed as the topological dual space of $L^{1}(\calF)$.

%
%%%%%%%%%%%%%%%%%%%%
\subsection{Velocity field as solution of an elliptic problem}

The velocity field can be  recovered  by solving the following div-curl system:
\begin{subequations} \label{elll}
\begin{align}
\div v = & \, 0 && \text{ in } \calF, \label{elll1} \\
\curl v = & \, \omega && \text{ in } \calF,   \label{elll2}\\ 
v\cdot n = & \, g && \text{ on } \partial \calF, \label{equ:bio:sav:ss} \\
 \label{defC} \int_{\pS^{i}} v(t,.) \cdot \tau = & \, \calC_{i}(t) && \text{ for } i \in \calI ,   
\end{align}
\end{subequations}
where we denote by $ \tau $ the counterclockwise tangent vector to the boundary. 
The quantities in the last equation are the circulations of the velocity vector field $v$ around the connected components $\pS^{i}$. 
 The reason why these circulations are important in the discussion is linked to the multiply-connectedness of the fluid domain $\calF$, this will be detailed below, in Section \ref{sec-dec}, after the analysis of the dynamics of these circulations.

%%%%%%%%%%%%%%%%%%%%
\subsection{Dynamics of the circulations around the sources and sinks}

For each $ i \in\calI $, the circulation $\calC_{i}(t) $
evolves in time according to the following Cauchy problem: 
\begin{equation}
\label{KLSS}
\calC_{i}'(t) = - \int_{\pS^{i}} \omega(t,.) g(t,.) ds, \quad  \calC_{i}(0) = \int_{\pS^{i}} v^{in} \cdot n ds .
\end{equation}
This follows from \eqref{fir:sy:ss-eq}  recast as 
\begin{equation}  \label{defE}
\partial_t v +  \omega v^\perp +   \nabla ( p + \frac12  \vert v  \vert^2) =  0 , 
\end{equation}
Indeed, by  \eqref{defC} and  \eqref{defE}, 
\begin{equation}
\label{KJJ}
\calC_{i}' = \int_{\pS^{i}} (\partial_t  v) \cdot \tau 
= - \int_{\pS^{i}}   \omega    \tau   \cdot v^\perp
- \int_{\pS^{i}} \tau  \cdot   \nabla ( p + \frac12  \vert v  \vert^2) .
\end{equation}
Since $\pS^{i}$ is a closed curve, 
the second integral is zero whereas the first one can be converted in the right hand side of the first equation in  \eqref{KLSS} by observing that 
$ \tau   \cdot v^\perp = - n \cdot v = -g$.
On the other hand the second equality in \eqref{KLSS} is another compatibility condition for the initial data $v^{in} $ with the boundary conditions. The identities in  \eqref{KLSS}  are known at least since the paper of Yudovich mentioned above, see \cite[Lemma 1.2]{Yudo}. 
By integration in time of \eqref{KLSS}, we arrive at the following formula for the circulations at time $t$: 
\begin{equation}
\label{bou:equ:ss}
\calC_{i}(t) = \calC_{i}^{in} - \int_{0}^{t} \int_{\pS^{i}} \omega^+ g \ \text{ for } i \in \calI^+, \   \text{ and } \
\calC_{i}(t) = \calC_{i}^{in} - \int_{0}^{t} \int_{\pS^{i}} \omega^- g \  \text{ for } i \in \calI^-. 
\end{equation}
Above we have separated the circulations around the sources and the ones around the sinks because the first ones are deduced from the boundary data $\omega^+ $ (the entering vorticity)  and $g$ (the entering normal velocity); they are therefore themselves to be considered as prescribed data for this problem. 
On the other hand 
 the second ones are unknowns of the problem since they involve the exiting vorticity that is the trace $\omega^{-} $ of $ \omega$ on $\calF_-$. 
 
 Moreover a computation  similar  to \eqref{KJJ}  for the circulation 
\begin{equation} \label{bouse1}
\calC_{\partial \Omega} = \int_{\partial \Omega} v(t,.) \cdot \tau , 
\end{equation} 
  around  the external boundary $ \partial \Omega $ proves that $\calC_{\partial \Omega} $ is constant in time, because of the impermeability condition on $ \partial \Omega $. 
This is the standard case of  Kelvin's theorem. 
Finally, by integration of the equation   \eqref{elll2} over the whole fluid domain ${\calF}$, we obtain the following identity, which holds at any time $t$, 
\begin{equation}  \label{bouse2}
\int_{\calF} \omega(t,.) = \calC_{\partial \Omega} + \sum_{i \in \calI}\calC_{i}(t).
\end{equation}
Therefore the circulation $\calC_{\partial \Omega}$ does not contain any new information and will not intervene in the sequel.

%%%%%%%%%%%%%%%%%%%%
\subsection{Decomposition of the velocity}
\label{sec-dec}

Let us first consider the following potential lift of the boundary data $g$ for the normal velocity: with any smooth enough $g$ we associate 
 $ v_g = \nabla \varphi $, where $\varphi_g $ is the unique solution of 
\begin{equation*}\begin{cases}  - \Delta \varphi_g = 0 \quad & \text{ in } \calF, \\ \nabla \varphi_g\cdot n = g \quad & \text{ in } \partial \calF .
 \end{cases} \end{equation*}
The regularity of the vector field $ v_{g} $ depends on the boundary data $ g $. 
For  more in this direction we refer for example to  \cite{GAGLI} and  \cite{KMPT}.

Let us also recall that for any smooth function $\omega$ there is a unique vector field 
 $ K_{H}[\omega]  $ satisfying 
\begin{subequations} \label{elll-KH}
\begin{align}
\div K_{H}[\omega]   = & \, 0 && \text{ in } \calF, \\
\curl K_{H}[\omega]   = & \, \omega && \text{ in } \calF,   \label{elllKH2}\\ 
K_{H}[\omega]  \cdot n = & \, 0 && \text{ on } \partial \calF, \label{equ:bio:sav:ss-KH} \\
 \label{defCK} \int_{\pS^{i}} K_{H}[\omega]   \cdot \tau = & \, 0 && \text{ for } i \in \calI .   
\end{align}
\end{subequations}
The mapping $\omega \mapsto K_{H}[\omega]  $  is called the hydrodynamical Biot-Savart law. It can be written as
an integral operator of the form 
\begin{equation}
\label{KHL}
K_{H}[\omega] (x) =  \int_{\calF} K(x,y) \omega (y) \, dy .
 \end{equation}
Moreover it follows from the Hodge-De Rham theory that the
 vector space  of the vector fields $v$ satisfying 
$ \div v = 0 $ and $\curl v =0$ in $ \calF$, and 
 $v\cdot n =0$ on $ \partial \calF$ is  of dimension $N$ (i.e. the number of holes in the fluid domain), and a basis of this  vector space is given by the 
 unique vector fields $ (X_i )_{i \in \calI}  $ satisfying 
\begin{subequations} \label{elllX}
\begin{align}
\div X_i = & \, 0 && \text{ in } \calF, \\
\curl X_i = & \, 0&& \text{ in } \calF,   \label{elll2X}\\ 
X_i\cdot n = & \, 0 && \text{ on } \partial \calF, \label{equ:bio:sav:ssX} \\
 \label{defCX} \int_{\pS^{j}} X_i  \cdot \tau = & \,   \delta_{ij}&& \text{ for } j \in \calI ,   
\end{align}
\end{subequations}
where the notation $ \delta_{ij} $ stands for the Kronecker symbols.
In view of the   \textit{a priori}   bounds \eqref{apq} and \eqref{apinf}
some important estimates regarding the operator $\omega \mapsto  K_{H}[\omega] $ 
 are the following: for $ 1 < q  < \infty$, there exists a constant $C>0$ such that 
\begin{equation} \label{CZ}
\|  K_{H}[\omega]  \|_{W^{1,q}(\calF)}
\leq C 
\| \omega  \|_{L^q(\calF)} ,  
\end{equation}
and  a constant $C>0$ such that 
\begin{equation} \label{CZ-infty}
\|  K_{H}[\omega]  \|_{LL(\calF)}
\leq C 
\| \omega  \|_{L^{\infty}(\calF)} ,  
\end{equation}
 Moreover for any smooth functions $g$ and $\omega$, there is a unique solution $v$ to \eqref{elll} and $v$ can be decomposed into 
\begin{equation}
\label{u:dec:ss}
 v = v_g +\sum_{i\in \mathcal{I}} \calC_i(t)   X_i + K_{H}[\omega] .
 \end{equation}
We refer here to \cite{Kato,Lin,flucher-gustafsson,MP} for more.

  %%%%%%%%%%%%%%%%%%%%
\subsection{Formal vorticity formulation}

Gathering  \eqref{tr-vort}, \eqref{bou:equ:ss} and \eqref{u:dec:ss}, we deduce 
 that the system  \eqref{fir:sy:ss}
  is formally equivalent to following vorticity-based reformulation: 
\begin{subequations} \label{vort-ss2}
\begin{align}
\partial_t \omega + v \cdot \nabla \omega = & \, 0 && \text{ in } \bbR^{+} \times \calF, \label{vort-ss2-a}  \\
\omega = & \, \omega^+ && \text{ on } \bbR^{+} \times \partial \calF^{+}, \label{vort-ss2-b}  \\
\omega(0,.) = & \, \omega^{in} && \text{ in } \calF ,\label{vort-ss2-c}  \\
v =& \, v_g +\sum_{i\in \mathcal{I}} \calC_i(t)   X_i + K_{H}[\omega] &&  \text{ in } \bbR^{+} \times \calF, \label{vort-ss2-d}  \\
 \calC_i(t)   = &\,  \calC_{i}^{in} - \int_{0}^{t} \int_{\pS^{i}} \omega^\pm g && \text{ for } i \in \calI^\pm \label{vort-ss2-e}.
\end{align}    
\end{subequations}
A few comments are in order. 
\begin{itemize}
\item 
Let us insist of the fact that there are two unknowns to the system  \eqref{vort-ss2}
 which are  $\omega $ and $\omega^{-}  $. In particular this is an additional feature of the present setting where the fluid exits through some holes in the domain that the exiting vorticity 
  $\omega^{-}  $ is necessary to  determine the velocity vector field $v$, see \eqref{vort-ss2-e}. 
  Let us mention that in the case, which is not considered in this paper,  where the fluid  exits from the domain only through a part of the external boundary $\partial \Omega$ then the determination of  the velocity $v$ inside the fluid domain $ \calF$ 
   is decoupled from  the exiting vorticity $\omega^{-}$, and so is the the determination of  the vorticity $\omega$ inside the fluid domain $ \calF$. 
 \item 
 Above the discussion has been quite formal; we did not care about the regularity or kind of solution for which the equivalence of 
 the system  \eqref{fir:sy:ss} and of the system  \eqref{vort-ss2} holds true. 
 Indeed the formulation \eqref{vort-ss2} seems much more appropriate to formulate rigorous mathematical results on the  problem at stake. 
 In this direction, it is worth to highlight that the velocity vector field $v$ being divergence free in $\calF$, the first equation of  \eqref{vort-ss2} can be rewritten in the conservative form: 
 \begin{equation}
  \label{conservative}
\partial_t \omega +  \div ( \omega v) =  0 \quad   \text{ in } \bbR^{+} \times \calF .
\end{equation}
 \end{itemize}

%%%%%%%%%%%%%%%%%%%%
\section{A first glance on the main results}

To avoid the reader to wait too long for an exposition of the main results of this paper, we first state the following informal statement which offers in a single glance some assertions regarding  the existence of solutions of  some different appropriate weak formulations of the system \eqref{vort-ss2} with input vorticities in $L^p$. 
\begin{Theorem}
\label{th-rough}
Let $p$ in $[1,+\infty]$ and  $(\omega^{in} , \omega^{+}  )$ a couple of input vorticities (\ref{CIV}) in $ L^p$.
Then there is $(\omega , \omega^{-} )$ a continuous couple (\ref{CC}) with values in $ L^p $ solution of the system \eqref{vort-ss2}. This solution have to be understood in different ways depending on the range of $p$ according to the following cases:
\begin{enumerate}[(i)]
\item for $p $ in $( 4/3,+\infty ]$, there exists a solution in a distributional sense,
\item  for $p$ in $(1,+\infty ]$, there exists a solution in a renormalized sense,
\item \label{point3}  for $p =1$,  there exists a solution of a symmetrized formulation. 
\end{enumerate}
In each of these cases, these solutions can be obtained as vanishing viscosity limits. 
\end{Theorem}
These results on the existence of weak solutions complement the existence (and uniqueness) of smooth (typically with vorticity in $W^{2,p}$) solutions obtained in  in the work  \cite{Yudo} by Yudovich. 
We refer to Theorem \ref{def:wea:sol:ss}, Theorem \ref{exi:Lp:ss}, and Theorem \ref{exi:L1:ss}  for precise statements corresponding  respectively to the three cases 
above. 
Yet, let us already state  a few remarks. 
\begin{itemize}
\item In the part (i) of Theorem \ref{th-rough}  we refer to the existence of solutions of a weak formulation with test functions supported up to the boundary, making a slight abuse of language by using the terminology of distributional solutions. 
Indeed such an existence result has already been proved in  \cite{MU} by Mamontov and Uvarovskaya. Their proof makes  use of smooth solutions of  \eqref{vort-ss2},  given  in the work  \cite{Yudo} by Yudovich, corresponding to regularized input data. 
Here we will take a slightly different path by considering some parabolic approximations, that is some Navier-Stokes type equations with vanishing viscosity. This suggest that these solutions are perhaps more physical. However we consider an artificial boundary condition for these parabolic approximations. A further step in the direction of the construction of more physical solutions could be to consider more physical boundary conditions for such viscous approximations. 
In this direction, let us mention the papers \cite{CA} and \cite{CC}, where some Navier slip-with-friction boundary conditions are considered.

\item In the  part (ii) of Theorem \ref{th-rough}  we refer to the renormalization theory  as initiated by Di Perna and Lions in  \cite{DL} for the transport equations. 
Moreover the part (ii) of Theorem \ref{th-rough} can be seen as an extension of the result \cite{CS} by Crippa and Spirito in the case of the two-dimensional incompressible Euler equations without any source nor sink. It extends the  part (i) in the sense that for $p $ in $( 4/3,+\infty ]$, the two types of solutions: distributional and renormalized, are equivalent. 

\item In  the part (iii) of Theorem \ref{th-rough}, we refer to a  weak reformulation of the problem  inspired by the works 
  \cite{Del} and \cite{S} respectively by Delort and Schochet where the case  of a diffuse positive Radon measure as  initial vorticity is addressed, in the case without any source nor sink. 
A crucial point in these works is that an energy estimate allows to prevent from a vorticity concentration in Dirac masses at some positive times. Such an argument seems difficult to reproduce in our setting, what leads us to deal with the easier case where the vorticity is $L^1$, for which an argument of propagation of compactness can be used to prevent from any concentration.  However, even with this restriction, some difficulties appear with the boundary conditions. 
 Indeed  the  symmetrized formulation hinted here  suffers from a loss of information regarding the prescription of the entering vorticity: only the vorticity fluxes integrating on the whole boundary of each sink are encoded in the formulation, not their pointwise values, see Remark \ref{loss} below. Let us  therefore highlight that the part (iii) of Theorem \ref{th-rough} has to be seen only as a partial result in the $L^1$ case.

\item In the three cases, the regularity of the velocity field is enough to 
 give a sense to the circulations 
 \eqref{defC}
 and
 \eqref{bouse1}. 
 Moreover,  for the solutions which are constructed in the proof of 
Theorem \ref{th-rough}, the time evolution of the circulations around the sources and sinks  is given by 
  \eqref{vort-ss2-e} whereas the circulation around the external boundary is constant. 
 Finally the conservation law
  \eqref{bouse2} holds at any time.

\end{itemize}

\begin{Remark} \label{WP}
 The issue of the uniqueness  of  weak solutions to the system above is a delicate topic, which requires some different types of argument. 
 It is the object of current investigations. 
\end{Remark}

%%%%%%%%%%%%%%%%%%%%
%%%%%%%%%%%%%%%%%%%%
\section{Precise existence results of distributional and renormalized solutions}
\label{sec-precis}

In this  section we precisely state the  existence results of distributional and renormalized solutions: the parts (i)  and (ii) of Theorem \ref{th-rough}. 
 
 \subsection{Existence of distributional solutions}
 Let us start with Part (i), that is  the existence of distributional solutions with $ L^{p}$  vorticity when $ p > 4/3 $.
 The terminology ``distributional solutions'' refers to the transport equation for the vorticity: we will require it to be satisfied in the sense that 
   for any $ \varphi  $ in $C^\infty_c ([0,+\infty) \times \overline{\calF} ; \R)$, 
\begin{align}
\label{wf:equ:ss:noren}
\int_{\calF} \omega^{in}\varphi(0,.) dx + \int_{\bbR^{+}}  \int_{\calF} \omega (\partial_t \varphi +  v \cdot \nabla \varphi ) \,  dx  \, dt = \, &\int_{\bbR^+}\int_{\partial \calF^{+}} g  \omega^+ \varphi ds dt \\ & + \int_{\bbR^+}\int_{\partial \calF^{-}} g \omega^{-} \varphi ds dt. \nonumber
\end{align}
For any test function $ \varphi  $ in $C^\infty_c ([0,+\infty) \times \overline{\calF} ; \R)$, the 
 equation \eqref{wf:equ:ss:noren}
 is obtained  from the equation  \eqref{conservative} by multiplying it by $\varphi  $ and integrating by parts taking into account the boundary conditions: $v\cdot n =  g $ on $\partial \calF$ and 
$ \omega = \omega^+$ on $ \partial \calF^{+}$. As already mentioned the function $\omega^{-}$ is an unknown, in particular because the existence of a trace on $ \partial \calF^{-}$
of a function $\omega$ which is only in $ L^\infty \left(\bbR^{+};  L^p (\calF)\right)$ does not follow from standard trace theorems. On the other hand let us highlight that if 
 $ \omega$ is a smooth solution of the transport equation  \eqref{conservative} then an integration by parts provides the 
 equation \eqref{wf:equ:ss:noren} with the trace of $\omega$ on $ \partial \calF^{-}$ instead of  $\omega^{-}$. 
The result that is alluded to in Part (i) of Theorem \ref{th-rough} is the following. 
\begin{Theorem}
\label{def:wea:sol:ss}
Let $ p $ in $(4/3,\infty] $. Let $ \calC_{i}^{in} $ in $\bbR$ for each $i$ in $\calI $ the initial circulations around $ \pS^i$.
Let  $  g$ a source-sink compatible function (\ref{SSC:g}) in $L^{1}_{loc}(\bbR^{+}; W^{1-1/p,p}(\partial \calF)) $   in the case where  $p \in (1, \infty )$ and in  $L^1_{loc}(\bbR^{+};W^{1,\infty}(\partial \calF))$ in the case where $ p = \infty$.
Let $(\omega^{in} , \omega^{+}  )$ a couple of input vorticities (\ref{CIV}) in $ L^p$.
  Then there exists $(\omega , \omega^{-}  )$ a continuous couple (\ref{CC}) with values in $L^p $  
such that  for any  $ \varphi  $ in $C^\infty_c ([0,+\infty) \times \overline{\calF} ; \R)$, 
the identity \eqref{wf:equ:ss:noren} is satisfied 
with $v$ given by  \eqref{vort-ss2-d} and \eqref{vort-ss2-e}, satisfying \eqref{apq} with equal sign and with $q=p$ in the case where $p<+\infty$ and \eqref{apinf}  in the case where $p=+\infty$.
\end{Theorem}
As already mentioned above this result has been obtained by  Mamontov and Uvarovskaya, see \cite{MU}. 
However we will provide in Section \ref{cas43}  a  slightly different proof based on some viscous approximations which are introduced in Section \ref{sec-va}.

 \subsection{Existence of renormalized solutions}
 
Now let us turn our attention to the part (ii) of Theorem \ref{th-rough}, that is to the existence of  renormalized solutions with $ L^{p}$  vorticity when  $ p > 1 $. 
It is based on the observation that for smooth vorticities satisfying  the transport equation  \eqref{vort-ss2-a} and for any function $ \beta \in C_{b}^1(\bbR) $, by the chain rule,
\begin{equation*}
\partial_t \beta(\omega) + v \cdot \nabla \beta(\omega) =  0 .
\end{equation*}
Then, for  any test function $ \varphi  $ in $C^\infty_c ([0,+\infty) \times \overline{\calF} ; \R)$,  by multiplying the previous equation  by $\varphi  $ and integrating by parts taking into account the boundary condition: $v\cdot n =  g $ on $\partial \calF$ and 
$ \omega = \omega^+$ on $ \partial \calF^{+}$, we arrive at 
\begin{align}
\label{wf:equ:ss:ren}
\int_{\calF} \beta(\omega^{in})\varphi(0,.) dx + \int_{\bbR^{+}}  \int_{\calF} \beta(\omega) (\partial_t \varphi +  v \cdot \nabla \varphi ) dx dt = \, &\int_{\bbR^+}\int_{\partial \calF^{+}} g \beta(\omega^+) \varphi ds dt \\ & + \int_{\bbR^+}\int_{\partial \calF^{-}} g \beta(\omega^{-}) \varphi ds dt. \nonumber
\end{align}
The terminology ``renormalized solution'' precisely refers to a vorticity satisfying  such identities. 
\begin{Theorem}
\label{exi:Lp:ss}
Let  $ p \in (1,\infty]$. Let $ \calC_{i}^{in} \in \bbR$ for $i\in \calI $ the initial circulations around $ \pS^i $.
 Let $g \in L^{1}_{loc}(\bbR^{+}; W^{1-1/p,p}(\partial \calF))$ in the case where  $p \in (1, \infty )$, and $g \in L^1_{loc}(\bbR^{+};W^{1,\infty}(\partial \calF))$ in the case where $ p = \infty$.
Assume that $g$ is source-sink compatible, see (\ref{SSC:g}).
 Let   $(\omega^{in} , \omega^{+}  )$ a couple of input vorticities (\ref{CIV}) in $L^p$.
Then there exists  a continuous couple $ (\omega, \omega^-) $ with values in $ L^p $ (see  (\ref{CC}))  is a renormalized solution, that is,  it satisfies \eqref{wf:equ:ss:ren}
for  any test function $ \varphi  $ in $C^\infty_c ([0,+\infty) \times \overline{\calF} ; \R)$, 
with $v$ given by   \eqref{vort-ss2-d} and \eqref{vort-ss2-e}. Moreover it satisfies 
    \eqref{apq} with equal sign and with $q=p$  for $ p < \infty $ and \eqref{apinf} for $ p = \infty$.
\end{Theorem}
The proof of Theorem  \ref{exi:Lp:ss} is given in Section \ref{sec-pr-ren}.

%%%%%%%%%%%%%%%%%%%%

\section{Precise existence results of symmetrized solutions}
\label{SF}

In this section we give a precised statement of the   point \eqref{point3}  of Theorem \ref{th-rough}. 
 We first recall the symmetrization argument, which has appeared in the works 
 \cite{T}, \cite{Del} and \cite{S}, and which leads to the formulation hinted in the point \eqref{point3}  of Theorem \ref{th-rough}.
The starting point of this argument is to recast the nonlinear term 
 $\omega  v$ of \eqref{conservative}  in a weak sense by using the integral expression of $v$ in terms of $\omega$.

\subsection{Case of a full plane}
 To recall this idea in a simple way let us consider the case where $v$ is given in terms of $\omega$
 by the usual Biot-Savart law in $\R^2$: 
\begin{equation}
\label{KR2}
v (x) =  \int_{\R^2}  K_{\R^2} (x,y) \omega (y) \, dy \quad   \text{ where } K_{\R^2} (x,y) := \frac{(x-y)^\perp  }{ 2\pi \vert x-y \vert^2}  .
 \end{equation}
Observe that we have dropped here the time variable  to simplify the exposition of the core of the argument for which it only plays the role of a parameter. 
Then for any test function $ \varphi \in C^{\infty}_{c}( \R^2) $, 
\begin{align} \nonumber
  \int_{\R^2}  \omega (x) v(x) \cdot \nabla \varphi (x) \, dx &=  \int_{\R^2}\int_{\R^2} K_{\R^2} (x,y) \cdot  \nabla \varphi (x) \omega (x)   \omega (y)  \, dx  \, dy 
  \\  \label{montagny}
  &= \frac12  \int_{\R^2}\int_{\R^2} K_{\R^2} (x,y) \cdot ( \nabla \varphi (x) - \nabla \varphi (y)) 
    \omega (x)   \omega (y)  \, dx  \, dy ,
 \end{align}
by symmetrization. The interest of the symmetrization is that, when $x$ and $y$ are close, one has, by Taylor's expansion, 
\begin{equation}
\label{justaylor}
\nabla \varphi (x) - \nabla \varphi (y) \sim D^2  \varphi (x) (x-y),
 \end{equation}
 and therefore the term 
\begin{equation}
\label{alpes}
 K_{\R^2} (x,y) \cdot ( \nabla \varphi (x) - \nabla \varphi (y))
 \sim \frac{1 }{ 2\pi} D^2  \varphi (x) \left( \frac{x-y}{\vert x-y \vert },  \frac{(x-y)^\perp}{\vert x-y \vert }\right)  ,
 \end{equation}
 remains bounded for $(x,y)$ in $\R^2 \times \R^2$. As a consequence the right-hand-side of \eqref{montagny} makes sense for a vorticity $\omega$ in $L^1 (\mathcal F) $.

\subsection{Case of a general domain}
\label{genecase}
For a general domain $\mathcal F$, by  \eqref{KHL},  and using a symmetrisation with respect to $x$ and $y$ as above, 
we obtain that,  for any test function $ \varphi \in C^{\infty}( \mathcal F) $, 
\begin{equation}
\label{lasalle}
\int_{\calF} \omega  K_{\calF}[\omega] \cdot \nabla \varphi = 
\int_{\calF} \int_{\calF}H_{\varphi}(x,y)\omega(t,x)\omega(t,y) \, dx \, dy,
\end{equation}
where $H_\varphi$ is the following  auxiliary function
\begin{equation}
\label{auxH}
 H_\varphi(x,y) = \frac{1}{2} \Big( \nabla_x  \varphi (x) \cdot K (x,y) + \nabla_y\varphi(y)\cdot K(y,x)\Big).
\end{equation}
It is classical, see \cite{Bramble,Eidus,gilbarg-trudinger}, that 
 there exists a constant $C$ such that for every $x, y \in \calF$, 
\begin{equation} \label{acheck}
\vert K (x, y) \vert
  \leq 
    \frac{C}{\abs{x - y}}.
\end{equation}
Moreover, since in the case of the full plane the counterpart of the function  $H_\varphi$ 
is half  the function given in the left hand side of \eqref{alpes}, which is bounded, one could wonder whether in the general case  
$ H_\varphi$ is bounded or not. 
Since $ H_\varphi$  can be decomposed as 
\begin{equation}
\label{decompH}
 H_\varphi(x,y) = \frac{1}{2}  \nabla_x  \varphi (x) \cdot  \Big(  K (x,y) + K(y,x)\Big)
 -
\frac{1}{2}  \Big(  \nabla_x  \varphi (x) - \nabla_y\varphi(y) \Big) \cdot K(y,x),
\end{equation}
where the second term is bounded 
thanks to \eqref{justaylor} and \eqref{acheck}, the question reduces to determine whether  
$ K (x,y) + K(y,x)$ is bounded or not. 
Indeed in the interior of the domain, the desingularization still occurs because the Biot-Savart law associated with any domain 
 is, away from the boundary, a regular perturbation of the Biot-Savart law associated with the full plane and given in \eqref{KR2}. 
 More precisely $ K (x,y) $ can be decomposed as 
\begin{equation}
\label{et}
   K (x,y) =  K_{\R^2} (x,y) + R(x,y), 
\end{equation}
with $R$ smooth in the interior set $\calF \times \calF$. 
Thus 
\begin{equation*}
     K (x,y) + K(y,x)   = R(x,y) + R(y,x) ,
\end{equation*}
is bounded on any compact subset of $\calF \times \calF$. 
As a consequence the right-hand-side of \eqref{lasalle} makes sense for a vorticity $\omega$ in $L^1 (\mathcal F) $ and for any test function $ \varphi \in C^{\infty}_{c}( \mathcal F) $.

%\subsection{ Boundary estimate} % in the case of a half-plane}
On the other hand, close to the boundary, some caution is needed. To illustrate the difficulty at stake, let us first consider the case where the fluid occupies a half-space. 
The Green function associated with the Laplace operator  in the right half-plane
% $\mathbb{R}^2_+ = \big\{ x = (x_1, x_2) \in \R / \quad  x_2 >0\big\}$
 with the Dirichlet condition on 
$ \R \times  \{  0 \}$ 
 is given for any $x = (x_1, x_2)$ and $ y = (y_1, y_2) $ in ${\mathcal F}$ by
\begin{equation*} 
  %   G_{\mathbb{R}^2_+} (x,y)  =   
      \frac{1}{4\pi}
        \ln \left(1 + \frac{4 x_2 y_2}{\abs{x-y}^2}\right) .
  \end{equation*}
The gradient of this function with respect to its first variable, is then given by
\begin{equation*} 
 %     \nabla G_{\mathbb{R}^2_+}(x,y)    = 
      \frac{1} {\pi (\abs{x - y}^2 + 4 x_2 y_2)}   \Bigl((0, y_2) - 2 x_2 y_2 \frac{x - y}{\abs{x - y}^2}\Bigr)    ,
  \end{equation*}
and its symmetrization is therefore
\begin{equation}
  \label{testco}
   %     \nabla G_{\mathbb{R}^2_+}(x,y)      +        \nabla G_{\mathbb{R}^2_+}(y,x)  = 
      \frac{(0, x_2 + y_2)}{\pi(\abs{x - y}^2 + 4 x_2 y_2)}.
  \end{equation}
  On the one hand the  tangential component vanishes, on the other hand the normal component is singular when $x$ and $y$ are close to each other and to the boundary. 
  Thus the cancellation observed in the case of the full plane is maintained in the case of  a half-plane up to the boundary for the tangential component but not for the normal one.
  Looking back to \eqref{decompH} we observe that for any test function  $ \varphi $ in $C^{\infty} ( \overline{  \mathcal F}) $ with $ \nabla  \varphi$ normal to the boundary, 
  the cancellation observed above in the case of the full plane is maintained and 
  the auxiliary function  $ H_\varphi$ is again  bounded, whereas it is not the case for a general test function $ \varphi $ in $C^{\infty} ( \overline{  \mathcal F}) $.
  %, because of the tangential part of its gradient. 

%\subsection{Boundary estimate in the case of a general domain}

This conclusion can be extended to any bounded domain with smooth boundary by using the mirror method, which is the  local approximation 
of the Green function of a general bounded domain associated with the Dirichlet  on the boundary  condition on $\partial \mathcal F$ by %
%satisfies  
$$
G(x,y)=\frac{1}{2\pi}\ln\frac{|\overline{x}-y|}{|x-y|}+O(1),
$$
as $x,y\to x_0\in\partial\mathcal F$, in $C^1$, where $\overline{x}$ 
%$\overline{x}:=2p(x)-x$ 
is  the mirror image of $x$ through $\partial\mathcal F$. 
This  mirror image of $x$ is well-defined for $x$ sufficiently close to the boundary  $\partial \mathcal F$ by the formula $\overline{x}:=2p(x)-x$, where 
$p(x)$ is  the orthogonal projection of $x$ on $\partial \mathcal F$ that is the element of $\partial \mathcal F$ 
  which satisfies 
 $|p(x)-x|=d(x)$, where $d(x)$ denotes the  distance to the boundary. Moreover there is   a neighborhood $\mathcal V$ of  $\partial \mathcal F$ where $p$ is $C^1$ and its derivative can be explicitly computed  in terms of $d(x)$, of  the unit tangent vector to $\partial\mathcal F$ at $p(x)$,  and of the curvature of $\partial\mathcal F$ at $p(x)$.
 This allows to approximate the function $H_\varphi(x,y)$ in $C^0$ as $x,y\to x_0\in\partial\mathcal F$ and to conclude that the following property holds for the auxiliary function $H_{\varphi}(x,y)$ when the test function $\varphi$ is in the space $ \mathfrak{C}_{0}(\calF) $  of the  functions  in $ C^\infty(\overline\calF)$ which are constant on any connected component of the boundary  (with a constant depending on each connected component) and equal to $ 0 $ on $ \partial \Omega $. 

\begin{Lemma}  \label{Hbd}
For any test function  $\varphi $  in $ \mathfrak{C}_{0}(\calF) $, the function $ H_\varphi$ is 
bounded on $\overline\calF\times\overline\calF\setminus\{(x,x)\ ;\ x\in \overline\calF\} $.% and vanishes when $x$ or $y$ are on the boundary $\partial \calF$.
\end{Lemma}

Thanks to \eqref{u:dec:ss}, we can recast the nonlinear term 
 $\omega  v$ of \eqref{conservative}  in a weak sense for any time-dependent test function  $ \varphi  $ in $C^\infty_c ([0,+\infty); \mathfrak{C}_0(\calF))$, 
\begin{equation}\label{REF}
\int_{\calF} \omega v\cdot \nabla \varphi = \int_{\calF} \omega v_g \cdot \nabla \varphi  + \sum_{i}  \calC_i(t)  \int_{\calF} \omega X_i \cdot \nabla \varphi + \int_{\calF} \int_{\calF} H_{\varphi}(x,y)\omega(t,x)\omega(t,y) \, dx \, dy .
\end{equation}

\subsection{Symmetrized formulation with  $ L^{1}$  vorticity}

With the previous considerations in hand we are now ready to precise the sense in which the solutions are considered  in the point \eqref{point3}  of Theorem \ref{th-rough}.
Combining  \eqref{u:dec:ss}, \eqref{wf:equ:ss:noren}, \eqref{lasalle} and  \eqref{REF} 
we are led to the following definition. 
\begin{Definition}
\label{def-L1-sol}
We say that a weakly continuous couple $(\omega , \omega^{-}  )$ with values in $ L^1$, see   (\ref{WCC}),   is a symmetrized solution to \eqref{transp} 
 if for any  $ \varphi  $ in $C^\infty_c ([0,+\infty); \mathfrak{C}_0(\calF))$, 
\begin{align}
\label{wf:1:equ:ss}
\int_{\calF} \omega^{in}\varphi(0,.) dx & + \int_{\bbR^{+}}  \int_{\calF} \omega \big( \partial_t \varphi + v_g \cdot \nabla \varphi \big) \, dx dt
 +  \sum_{i}  \int_{\bbR^+}  \calC_i(t)  \int_{\calF} \omega X_i \cdot \nabla \varphi  \, dx \, dt \\ & 
 +  \int_{\bbR^+}\int_{\calF} \int_{\calF}H_{\varphi}(x,y)\omega(t,x)\omega(t,y) \, dx \, dy \, dt
  = \int_{\bbR^+}\int_{\partial \calF^{+}} g \omega^+ \varphi ds dt \nonumber
  \\ & \quad \quad  \quad \quad \quad   \quad \quad  \quad \quad \quad   \quad \quad  \quad \quad \quad  
  + \int_{\bbR^+}\int_{\partial \calF^{-}} g \omega^{-} \varphi ds dt,  \nonumber
\end{align}
where
\begin{equation}
\label{formuCi}
\calC_{i}(t) = \calC_{i}^{in} - \int_{0}^{t} \int_{\pS^{i}} \omega^+ g \ \text{ for } i \in \calI^+ \   \text{ and } \
\calC_{i}(t) = \calC_{i}^{in} - \int_{0}^{t} \int_{\pS^{i}} \omega^- g \  \text{ for } i \in \calI^-. 
\end{equation}
\end{Definition}

\begin{Remark} \label{loss}
Observe that the weak formulation \eqref{wf:1:equ:ss}, when restricted to test  functions  $ \varphi(t..) $ in $ \mathfrak{C}_0 $
suffers from a loss of information on the boundary data since it only depends on $ \omega^{+}$ through the integrals
\begin{equation}
\label{discord}
\int_{\partial \calF^{+}} g \omega^+ \, dz ,
\end{equation}
where the time  variable $t$ plays the role of implicit parameter and takes its values in $\R_+$. 
 In particular if one considers two sets of smooth data, for the initial and boundary conditions, for which the entering vorticities are distinct but with the same value for the total entering 
 vorticity flux given by \eqref{discord}, then by the existence and uniqueness result of Yudovich in  \cite{Yudo}  there are corresponding smooth solutions to the system  \eqref{vort-ss2}. These  two solutions  satisfy 
  the weak formulation \eqref{wf:1:equ:ss} for any test  functions  $ \varphi $ in $C^{\infty}_{c}([0,+\infty);\mathfrak{C}_0 (\calF))$, but are distinct since their respective traces on $\partial \calF^{+}$ are supposed to be different. 
  On the other hand, in this approach, this unfortunate to restriction of the test functions   $ \varphi(t..) $ to be in $ \mathfrak{C}_0 $ 
  seems mandatory  because of the discussion in Section 
  \ref{genecase}.
   \end{Remark}

%%%%%%%%%%%%%%%%%%%%
%\subsection{Weak formulation with symmetrized vorticity} \label{sec:w}
%for any $ \varphi \in C^{\infty}_{c}(\bbR^{+}\times \overline{\calF}) $, 

The precise statement hinted in the point \eqref{point3}  of Theorem \ref{th-rough} is the following. 

\begin{Theorem}
\label{exi:L1:ss}
 Let $ \calC_{i}^{in} \in \bbR$ for $i\in \calI $ the initial circulations around $ \pS^i $.
 Let $g \in L^{1}_{loc}(\bbR^{+}; L^{1}(\partial \calF))$  source-sink compatible, see (\ref{SSC:g}).
 %such that $ \int_{\partial \calF } g = 0 $, $ g = 0 $ on $ \partial \Omega $, $ g < 0 $ on $ \partial \calF^{+} $ and $ g > 0 $ on $ \partial \calF^{-} $. 
 % a couple of input vorticities  in $L^p$ in the sense of Definition  \ref{def-input}.
 Let  $(\omega^{in} , \omega^{+}  )$ a couple of input vorticities  (\ref{CIV}) in $L^1$. 
 Then there is a symmetrized solution  $(\omega , \omega^{-}  )$ to \eqref{transp} in the sense of Definition \ref{def-L1-sol}
 and   satisfying the inequality \eqref{apq} with $q=1$.
\end{Theorem}
The proof of Theorem  \ref{exi:L1:ss} is given in Section \ref{sec-proof-sym}.

%%%%%%%%%%%%%%%%%%%%%%%%%%%%%%%%%%%%

\section{Remainder on  the transport equation with given non-tangential velocity}
\label{sec-transport}

In this section we recall a few instrumental facts regarding the transport equation 
\begin{subequations} \label{transp}
\begin{align}
\partial_t \omega + v \cdot \nabla \omega = & \, 0 && \text{ in } \bbR^{+} \times \calF, \label{transp:1} \\
\omega = & \, \omega^+ && \text{ on } \bbR^{+} \times \partial \calF^{+},  \\
\omega(0,.) = & \, \omega^{in} && \text{ in } \calF ,
\end{align}    
\end{subequations}
where the \ref{SSC:v} vector field $v$ is assumed to be given and to satisfy the assumptions:
\begin{gather}
\label{hyp-v}
v  \in   L_{loc}^1 \left(\bbR^{+};  W^{1,q} (\calF)\right) \quad \text{ and } \quad \div v =0 \, \text{ in } \calF, 
\end{gather}
for some $q$ in $[1,+\infty]$.

This is a quite classical topic in the case where the boundary $ \partial \calF$ is impermeable.
In the case where the fluid can enter into and exit of the boundary  $ \partial  \calF $ of the domain, a nice reference is \cite{Boyer} where trace issues, as well as existence and uniqueness of solutions in various senses for the system \eqref{transp} is considered. Note that in \cite{Boyer}, the author assumes extra regularity of the normal component of the velocity on the boundary, namely $ v\cdot n \in L^{\alpha}_{loc}(\bbR^{+} \times \partial \calF) $ for some $ \alpha > 1 $. In our work this condition is removed thanks to the peculiar geometry at stake, as explained in the following remark.

 \begin{Remark}
\label{rem:4}
In this paper we will assume the velocity fields to be \eqref{SSC:v}, which implies that the regions where the normal component of the velocity field has distinguished signs are well-separated, in the sense that there exists a smooth cut-off $ \psi $ that is identically equal to $ 1 $ in a neighborhood of $ \partial \calF^+ $ and identically equal to $ 0 $ a neighborhood of  the remaining  boundary. With the help of this cut-off it is then possible to deduce information on $ \omega $ on $ \partial \calF^+ $, and on $ \partial \calF^-$,  from the value of $ \omega $ in $ (0,T) \times \calF $. To do that is enough to multiplying \eqref{transp:1} with the smooth cut-off $  \psi $ (or $ 1- \psi $), to integrate in $ (0,T) \times \calF $ and to do some integrations by parts. 
In the general case when the normal component of the velocity does not have a sign, it is not clear  how to define a cut-off that separates the regions where $ v\cdot n $ has different signs. In \cite{Boyer}, this difficulty is tackled by assuming that $ v\cdot n \in L^{\alpha}_{loc}(\bbR^{+} \times \partial \calF) $ for some $ \alpha > 1 $. This hypothesis is used in an essential manner in \cite[(3.14 )]{Boyer}.
Within our geometrical setting, with the separation  of the regions of the boundary where the normal component of the velocity field has different signs, the results showed in \cite{Boyer} are valid without the extra integrability assumption  $ v\cdot n \in L^{\alpha}_{loc}(\bbR^{+} \times \partial \calF) $ for some $ \alpha > 1 $.
\end{Remark}

 \subsection{Distributional and renormalized  solutions}
In this subsection we recall the definitions of distributional and renormalized  solutions to the  transport equation  \eqref{transp}. 
We recall first the following terminology.
\begin{Definition}
We say that $p$ and $q$ in $[1,+\infty]$ are conjugated if $\frac1p + \frac1q =1$.
\end{Definition}
Distributional solutions of the transport equations are then defined as follows. 
\begin{Definition}
\label{def-reno}
Let $p$ in $[1,+\infty]$, let $ v $ satisfying the hypothesis \eqref{hyp-v}, with $q$ such that $p$ and $q$ in $[1,+\infty]$ are conjugated, and let $(\omega^{in} , \omega^{+}  )$ a (\ref{CIV}) in $ L^p $.
We say that a continuous couple (\ref{CC}) with values in $ L^p$ $(\omega , \omega^{-}  )$ is a distributional solution to \eqref{transp} if for any 
$ \varphi  $ in $C^\infty_c ([0,+\infty) \times \overline{\calF} ; \R)$,   the identity  \eqref{wf:equ:ss:noren} is satisfied. 
\end{Definition}
On the other hand renormalized solutions of the transport equations  are  defined as follows whatever $p$ and $q$ are conjugated or not. 
\begin{Definition}
\label{def-no}
Let $p$ in $[1,+\infty]$, let $ v $ satisfying hypothesis \eqref{hyp-v} and $(\omega^{in} , \omega^{+}  )$ a (\ref{CIV}) in $ L^p $. 
We say that a continuous couple (\ref{CC}) with values in $ L^p$ $(\omega , \omega^{-}  )$ is a renormalized solution to \eqref{transp} if for any 
$ \beta \in C_b^1(\bbR) $ and  for any 
$ \varphi  $ in $C^\infty_c ([0,+\infty) \times \overline{\calF} ; \R)$,  the identity  \eqref{wf:equ:ss:ren} is satisfied.
\end{Definition}

 \subsection{Existence and uniqueness results}

The following result gathers some results regarding the existence and uniqueness of distributional and renormalized solutions to the transport equation  \eqref{transp}  and a duality formula.
\begin{Proposition}
\label{exi-pq}
Let $p$ and $q$ in $[1,+\infty] $, let $ v $ satisfying the hypothesis \eqref{hyp-v}, let  $(\omega^{in} , \omega^{+}  )$ a (\ref{CIV}) in $L^p$. 
Then we have the following results.
\begin{enumerate}[1)]
\item  If  $p$ and $q$ in $[1,+\infty]$ are conjugated, there exists a unique distributional solution $(\omega , \omega^{-}  )$ in $L^p$ to \eqref{transp} (in the sense of Definition 
\ref{def-reno}).
\item If  $p$ and $q$ in $[1,+\infty]$ are conjugated, any distributional solution $(\omega , \omega^{-}  )$ in $L^p$  to \eqref{transp}  (in the sense of Definition 
\ref{def-reno}) is a renormalized solution to \eqref{transp}
(in the sense of Definition 
\ref{def-no}).
\item There exists a unique renormalized solution $(\omega , \omega^{-}  )$ in $L^p$ to \eqref{transp} (whatever $p$ and $q$ are conjugated or not).
\item \label{exi-pqduality}
 If  $p$ and $q$ in $[1,+\infty]$ are conjugated, if 
$(\omega , \omega^{-}  )$ is a  $L^p$ renormalized solution  to \eqref{transp},
$\Psi $ is in $L^q ( [0,T] \times \partial \calF^{-} ;g \, dz \, ds)$, $\phi_{T}$ is in  $L^q ( \calF)$, $\chi $ is in $L^1  ( [0,T]  ;L^q( \calF))$, 
and $(\phi, \phi^+) $ is a  $L^q$ renormalized solution of the backward transport equation 
\begin{align*}
-\partial_t \phi -v \cdot \nabla \phi = & \, \chi \quad && \text{ in } [0,T] \times \calF, \nonumber \\
 \phi = & \, \Psi  \quad && \text{ on } [0,T] \times \partial \calF^{-},  \label{rrr} \\
\phi(T,x) = & \, \phi_{T}(x), &&  \nonumber
\end{align*}
then 
\begin{equation} \label{dua1}
    \int_0^{T}\int_{\calF} \omega \chi + \int_0^{T} \int_{\partial \calF^{-}} g \omega^{-} \Psi = \int_{\calF} \omega^{in} \phi(0,.) - \int_{\calF} \omega(T,.) \phi_{T} - \int_{0}^{T} \int_{\partial \calF^{+}} g \omega^{+} \phi^{+}.
\end{equation}
\end{enumerate}
\end{Proposition}

 \begin{proof}[Proof of $1)$]
 The case where  $p$ is in $(1,+\infty]$ is already proved in  \cite{Boyer}. In particular \cite[Th. 4.1]{Boyer}  and \cite[Th. 4.1]{Boyer} deal with respectively  $p=+\infty$ and $p$ in $(1,+\infty)$. 
 Moreover the case where $p=1$ can be tackled along the same way. 
 We briefly recall the sketch of the proof in \cite{Boyer} for sake of completeness. Regarding existence it is enough to consider a viscous approximation, quite similar to the one we will introduce in Section \ref{sec-va}%,  which satisfies  \eqref{hyp-vnu} 
  and to pass to the limit in the weak formulation from \textit{a priori}  bounds. On the other hand, at a formal level, uniqueness follows from the linearity of the transport equation and of the \textit{a priori}  estimates. To justify rigorously these steps it enough to consider the regularization of the solutions, see \cite[Sec. 2.2]{Boyer}, which satisfy pointwise almost everywhere the transport equation associated with the vector field $ v$  and with a source term that converge to zero in $ L^1 $. The $ L^1 $ \textit{a priori} bound for the difference of two regularized solutions implies uniqueness.   \end{proof}

 \begin{proof}[Proof of $2)$]
It is a consequence of point $ 2 $ of \cite[Th. 3.1]{Boyer}. Indeed, in this reference, a stronger result is proved, since  $ \omega $  is there only assumed to satisfy \eqref{transp} for test functions supported away from the boundary. Then the existence of traces on the boundary such that the weak formulation holds for test functions supported up to the boundary is proved as a consequence of the fact that  $ \omega $ satisfies the transport equation in the distributional sense inside. As presented in Remark \ref{rem:4}, in \cite{Boyer} the velocity field satisfies the extra hypothesis $ v\cdot n \in L^{\alpha}_{loc}(\bbR^{+} \times \partial \calF) $ for some $ \alpha > 1 $, which is replaced in our setting by the fact that the regions where $ v\cdot n  $ has different sign are well-separated.
  \end{proof}
 
  \begin{proof}[Proof of $3)$]
 The existence of a  renormalized solution $(\omega , \omega^{-}  )$ in $L^p$
 in the  case where  $1/p + 1/ q \leq 1 $ follows from points $ 1)$ and $ 2) $. 
 In the case where $ 1/p + 1/ q > 1 $, consider an injective function $ b \in C^1_b(\bbR) $.
  Since $ (b(\omega^{in}), b(\omega^+))\in L^{\infty}\cap L^1 $ and  $ 1/\infty + 1/ q \leq 1 $,  there exists $ (\mathfrak{b},\mathfrak{b}^-) $  a renormalized solution to the transport associated with the data $ (b(\omega^{in}), b(\omega^+)) $. 
  Using the injectivity of $ b $ we define $(\omega, \omega^{-} ) := ( b^{-1}(\mathfrak{b}),b^{-1}(\mathfrak{b}^-))$ and we observe that it is a renormalized solution to the transport equation. Let us now prove the  uniqueness part of the statement. Suppose  that there exists another renormalized solution $(\tilde{\omega}, \tilde{\omega}^-) $ associated with the same initial data.
  By definition, for the same function $b$, the couple $(b(\tilde{\omega}), b(\tilde{\omega}^-)) $ is a distributional solution to the transport equation. From uniqueness of point $ 1) $ we deduce that   $(b(\omega), b(\omega^-)) = (b(\tilde{\omega}), b(\tilde{\omega}^-))  $. Since  $ b$ is injective, this implies that $ (\omega, \omega^-) = (\tilde{\omega}, \tilde{\omega}^-) $.
  \end{proof}
 
 \begin{proof}[Proof of $4)$]
  At a formal level it is enough to test the weak formulation satisfied by $ (\omega,\omega^{-}) $ with $(\phi, \phi^+) $. To show this point rigorously, we proceed as previously by considering an appropriate regularization of  $ \phi $ with a smoothing kernel, see \cite[Sec. 2.2]{Boyer}, which in particular does not change boundary data. This regularization of  $ \phi $   satisfies the transport equation associated with the velocity field $v $ %pointwise almost everywhere
  and with right hand side a function that converge to $ \chi $ in $ L^1 $. Then we  consider the weak formulation of  the renormalized transport equation satisfied by $ (\omega, \omega^{-}) $ with as test function the previous regularization of $ \phi$ and we pass to the limit to obtain \eqref{dua1} for $ (b(\omega), b(\omega^{-}))$. Finally we  conclude by approximating $ b(x) = x $ via an appropriate approximation-truncation process.     
  \end{proof}

%%%%%%%%%%%%%%%%%%%
\section{Smooth viscous approximations}
\label{sec-va}

This section is devoted to an auxiliary system, a  transport-diffusion  equation with a small parameter in front of the diffusion term,  
 which is useful  to construct solutions to the transport equation \eqref{conservative}.

%%%%%%%%%%%%%%%%%%%%
\subsection{A family of viscous approximated models}
\label{sec-vcv}

We consider, for $\nu $ in $(0,1)$, 
\begin{subequations} 
\label{equ:app:ss}
\begin{align}
\partial_t \omega_{\nu} + v_{\nu}\cdot \nabla \omega_{\nu}  = & \,  \nu \Delta \omega_{\nu}  \quad && \text{ in  } \bbR^{+} \times \calF, \label{equ:app:ss1} \\   
  \omega_{\nu} (0,\cdot) = &  \, \omega^{in}_{\nu}  \quad && \text{ in }  \calF, \label{equ:app:ssin} \\   
\nu \partial_n \omega_{\nu}   = &  \, (\omega_{\nu} -\omega^{+}_{\nu})g_{\nu} \mathds{1}_{\partial F_{+}}  \quad && \text{ in } \bbR^{+}\times \partial \calF \label{equ:app:ssin2}, 
 \end{align}
 \end{subequations}
where  $ \omega^{in}_{\nu} $, $ \omega_{\nu}^{+} $ and $ g_{\nu} = v_{\nu} \cdot n $, are given. The notation $\mathds{1}_{\partial F_{+}} $ stands for the set function associated with the inlet $\partial F_{+}$, which is equal to $1$ when $x$ is in $\partial F_{+}$, and $0$ otherwise.

Regarding the family of \eqref{SSC:v} vector fields $(v_{\nu})_{\nu \in (0,1)}$, we will consider several cases:
\begin{enumerate}[(i)]
\item  the family is constant equal to a given vector field: there is a given vector field $v$ such that $v_{\nu} = v$ 
 for each $\nu $ in $(0,1)$, 
 \item the family $(v_{\nu})_{\nu \in (0,1)}$ is not constant but is given,
 \item  the family $(v_{\nu})_{\nu \in (0,1)}$ is related to the family of vorticity $(\omega_{\nu})_{\nu \in (0,1)}$ by \eqref{u:dec:ss}. 
\end{enumerate}
%
% Case (i) will be useful in ...
 In the two first cases we  assume that the sequence of vector fields 
 $(v_{\nu})_{\nu \in (0,1)}$ 
 satisfy the assumptions:
\begin{gather}
 \label{hyp-vnu}
  \div v_{\nu} =0 \, \text{ in } \calF \quad   \text{ and } \quad v_{\nu}\cdot n =  g_{\nu} \text{ on } \partial \calF.
\end{gather}

%\begin{gather}
% \label{hyp-vnu}
%v_{\nu}  \in   L^1_{loc} \left(\bbR^{+};  W^{1,q} (\calF)\right) ,\quad   \div v_{\nu} =0 \, \text{ in } \calF \quad   \text{ and } \quad v_{\nu}\cdot n =  g_{\nu} \in L^{r}_{loc}(\bbR^{+ } \times \partial \calF),
%\end{gather}
%
%$(v_{\nu})_{\nu \in (0,1)}$
%
% In the third case this property will be deduced from the regularity of $\omega_{\nu}$.

% some sequences of smooth approximations respectively of $ \omega^{in}$, $ \omega^{+} $ and $ g $ when 
%$\nu $ converges to $0$ with positive values. 
In the case (iii), the system \eqref{equ:app:ss}  is a Navier-Stokes type system with a non-standard boundary condition   which corresponds to a penalisation of the 
Yudovich's boundary conditions  \eqref{fir:sy:ss-Y2}.
A similar approximation system was used in \cite{Boyer}.  
The interest to choose such  boundary conditions appears when considering the weak formulation of the system \eqref{equ:app:ss}: assuming that $\omega_{\nu}$ is a smooth solution of  \eqref{equ:app:ss} and multiplying by a test function $ \varphi $ in $C^{\infty}_{c}(\bbR^{+}\times \overline{\calF};\R) $,  we obtain after some integrations by parts, using \eqref{hyp-vnu}, and some simplifications: 
\begin{gather}
\label{weak:for:SS:MB}
\int_{\calF} \omega^{in}_{\nu} \varphi(0,.) dx +
\int_{\bbR^{+}} \int_{\calF} \omega_{\nu} ( \partial_t   +  v_{\nu} \cdot \nabla ) \varphi   - \nu \int_{\bbR^{+}}\int_{\calF} \nabla \omega_{\nu} \cdot \nabla \varphi 
\\ \nonumber \quad \quad  = \int_{\bbR^{+}} \int_{\partial \calF^{+}} g_{\nu} \omega^+_{\nu} \varphi + \int_{\bbR^{+}} \int_{\partial \calF^{-}} g_{\nu} \omega_{\nu}  \varphi .
\end{gather}
Observe in particular that the integrand in the first integral in the right hand side contains the prescribed value $\omega^+_{\nu}$ rather than the trace of the unknown   $\omega_{\nu}$. 

%%%%%%%%%%%%%%

\subsection{Reminder on the weak compactness in $L^1$}
%on De la Vall\'ee Poussin's lemma}

One of the ingredients to deal with the case $ p = 1 $ is the so called De la Vall\'ee Poussin's lemma which establishes the equivalence of different definitions of uniform integrability. 

\begin{Lemma}[De la Vall\'ee Poussin's lemma]
\label{Delavalle:ain}
Let $ \mathcal{O} \subset \bbR^n $ bounded and measurable and let $ f_j : \mathcal{O} \longrightarrow \bbR $,  with $ j \in \mathbb{N} $,  a sequence of mesurable functions.
 The following definitions of uniform integrability are equivalent.
%statements are equivalent.

\begin{enumerate}

\item It holds $$\displaystyle{ \lim_{n \to \infty} \sup_{j\in \mathbb{N}} \int_{\{ |f_j|> n \}} |f_j | = 0}. $$
\item For any $ \eps > 0 $ there exists $  \delta > 0$ such that for any measurable $ A \subset \mathcal{O} $ with Lebesgue measure $ \mu(A) < \delta $,  
$$\sup_{j\in \mathbb{N}} \int_{A} |f_j| < \eps .$$ 

\item There exists a convex even smooth function  $ G: \bbR \longrightarrow \bbR^{+} $  such that 
\begin{equation*}
\lim_{|s| \longrightarrow + \infty} \frac{G(s)}{|s|} = + \infty \quad \text{ and } \quad \sup_{j \in \mathbb{N}} \int G(|f_j|) < + \infty. 
\end{equation*}

\end{enumerate}

\end{Lemma}

\begin{proof}
We recall how to prove 
 the implication $ 1. \Rightarrow 3. $, which is the most difficult one, to prepare the sequel. Assumption $ 1. $ implies the existence of an increasing sequence $ N_i $ with $ i \in \mathbb{N} $ such that $ N_0 = 0 $ and such that for any $ i \geq 1 $, 

\begin{equation} \label{ff}
\sup_{j \in \N} \int_{|f_j|> N_i } |f_j| < \frac{1}{2^i}.
\end{equation} 

We consider  the unique even function $G$  such that for any $ i \geq 1 $,
\begin{equation*}
G(s) = i N_i  + \frac{s-N_i}{N_{i+1}-N_{i}}((i+1)N_{i+1}-iN_{i}) \quad \text{ for } s \in [N_i,N_{i+1}).
\end{equation*}

Note that the above function is increasing, convex and superlinear. Moreover
\begin{align*}
\sup_{j \in \mathbb{N}} \int G(|f_j|) = \, & \sup_{j}\sum_{i}\int_{N_{i} \leq | f_{j}| < N_{i+1} } G(|f_j|) \leq  \sup_{j}\sum_{i}\int_{N_{i} \leq | f_{j}| < N_{i+1} } (i+1) |f_j| ,
\end{align*}
since $G(s) \leq (i+1) s$ for $ s $ in $[N_i,N_{i+1})$.
Thus, by \eqref{ff}, we arrive at 
\begin{align*}
\sup_{j \in \mathbb{N}} \int G(|f_j|) \leq & \, \sup_{j}\left(N_1 |\mathcal{O}| + \sum_{i}  \frac{i+1 }{2^i}\right)  < + \infty.
\end{align*}
To conclude it is sufficient to regularize 
 $ G $ by a suitable convolution process. 
 %it is not smooth. To solve this issue it is enought to covolute $ G $ with an appropriate smoothing kernel. Let $ \eta \in C^{\infty}_c(-1,0) $, $ \eta \geq 0 $ and $ \int \eta = 1 $, then $ \tilde{G} = \eta \star G_0 $ where $ G_0 $ is the extention by zero of $ G $ for negative entries. The function $ \tilde{G}(|x|) $ is even, convex, superlinear, moreover $ \tilde{G}(|x|) \leq G(|x|) $.

\end{proof}

The following corollary will be useful to deal with the incoming vorticity.

\begin{Corollary}
\label{Cor:DLVPoussin}
Let $ \mathcal{O} \subset \bbR^n $ bounded and measurable and let $ f_j, h_j: \mathcal{O} \to \bbR $,  with $ j \in \mathbb{N} $, two sequences of measurable functions such that $ h_j  > 0 $, the sequence $ h_j $ is uniformly integrable and
\begin{equation}
\label{Hip:Ain} 
\lim_{\delta \to 0 } \sup_{j \in \mathbb{N}} \mu\{ h_j \leq \delta \} = 0. 
\end{equation}
Then the following assertions are equivalent:
\begin{enumerate}[i.]

\item The sequence $ f_j h_j $ is uniformly integrable.     

\item There exists  an even convex function $ G: \bbR \to \bbR^{+} $ such that
\begin{equation*}
\lim_{|s| \longrightarrow + \infty} \frac{G(s)}{|s|} = + \infty \quad \text{ and } \quad \sup_{i}\int_{\mathcal{O}}G(f_j)h_j < + \infty. 
\end{equation*}  
\end{enumerate}

\end{Corollary}

\begin{proof}

\item
\paragraph*{$  i. \Rightarrow ii. $} By  \eqref{Hip:Ain} and assumption $ i. $, we have
\begin{equation}\label{arg1}
\sup_{j \in \N} \int_{\{ h_j \leq \delta \}} | f_j h_j | < \eps,
\end{equation}
where we use the definition 2. of uniform integrability from Lemma \ref{Delavalle:ain}.

Again from assumption $ i. $, we have 
\begin{equation}\label{arg2}
\lim_{n \to +\infty} \sup_{j} \int_{\{ |f_j h_j| > \delta n\}} |f_j h_j| = 0.
\end{equation}
Since 
\begin{equation*}
\left\{ f_j > n \right\}% = \left(\left\{ f_j > n \right\}\cap \left\{ h_j > \delta \right\} \right) \cup \left(\left\{ f_j > n \right\}\cap \left\{ h_j \leq \delta \right\}  \right)
 \subset \left\{ f_j h_j  > \delta n \right\} \cup  \left\{ h_j \leq \delta \right\} ,
\end{equation*}
it follows from \eqref{arg1} and \eqref{arg2} that 
\begin{equation*}
\lim_{n \to +\infty} \sup_{j \in \N} \int_{\{ |f_j| > n\}} |f_j h_j| = 0 .
\end{equation*}
Now to conclude the proof of Corollary 
\ref{Cor:DLVPoussin} it is sufficient to mimick 
the implication $ 1. \Rightarrow 3. $ of the previous proof of De la Vall\'ee Poussin's lemma.

\paragraph*{$  ii. \Rightarrow i. $}  Suppose by contradiction that $ i. $ is false. From definition 2. of uniform integrability from Lemma \ref{Delavalle:ain}, there exists $ \eps > 0 $ such that for any $ \delta > 0 $ there exists $ A_{\delta} \subset \mathcal{O} $ with $ \mu(A_{\delta}) < \delta $ and $ j_{\delta} = j   \in \mathbb{N} $ for which 
$$ \int_{A_{\delta}} | f_j h_j | > \eps. $$
Now from assumption $ ii. $ there exists an 
  increasing and convex function $ G $ for which we have, by the Jensen inequality, 
\begin{equation*}
G\left(\frac{\eps }{\int_{A_{\delta}}{|h_j |}} \right) \leq  G\left(\frac{\int_{A_{\delta}}{|f_j h_j |}}{\int_{A_{\delta}}{|h_j |}} \right)\leq \frac{\int_{A_{\delta}}{G(|f_j|) h_j }}{\int_{A_{\delta}}{|h_j |}},
\end{equation*}
and therefore 
\begin{equation}
\label{last:ain}
\frac{G\left(\frac{\eps }{\int_{A_{\delta}}{|h_j |}} \right)}{\frac{\eps }{\int_{A_{\delta}}{|h_j |}}} \leq \frac{C}{\eps}.
\end{equation}
with $C = \sup_{i}\int_{A_{\delta} }G(f_j)h_j < + \infty$. 
Note that $ \eps $ is fixed and from uniform integrability of the sequence $ h_j $, we have $ \int_{A_{\delta}}{|h_j |} \to 0 $ as $ \delta \to 0 $. 
Thus from assumption $ ii. $ the left hand side of \eqref{last:ain} tends to $ + \infty $ which is in contradiction to the fact that the right hand side is bounded.
\end{proof}

Let us also recall the Dunford-Pettis theorem which links the uniform integrability to the weak compactness in $L^1$, that is, the Dunford-Pettis theorem asserts that 
a subset of $L^1$ is  
weakly relatively compact if and only if it is uniformly integrable.

\begin{Remark} Let us stress that it follows from Egoroff's theorem and  the Dunford-Pettis theorem  that a sequence  $ (h_{j})_j $  such that $ h_j \to h $ in $ L^1(\mathcal{O}) $ with  $ h > 0 $ 
 satisfies the assumptions required in 
 the above corollary.
\end{Remark}

Moreover to address the issue of compactness in  $C([0,T]; L^1 - w)$ 
we recall the following variant of the Arzel\`a-Ascoli theorem,  see \cite[Theorem 1.3.2]{Vrabie}. 
\begin{Theorem}
\label{th-ascoli}
Let $X$ be a real Banach space.
Let $C([0,T]; X  - w)$ the locally convex topological vector space of the functions continuous on $[0,T]$ with values in $X$ endowed with its weak topology. 
Let $K$ a subset of $C([0,T]; X  - w)$. 
Then $K$ is relatively  sequentially compact in $C([0,T]; X  - w)$ 
 if and only if the two following assertions hold true: 
\begin{itemize}
\item $K$ is weakly equicontinuous on $[0,T]$, 
\item there exists a dense subset $D$ in $[0,T]$ such that for any $t$ in $D$, 
$K(t) := \{ f(t) / \quad f \in K \}$ is weakly relatively compact in $X$. 
\end{itemize}
\end{Theorem}
Observe that Theorem
\ref{th-ascoli} can be applied to some cases where the closed unit ball of $X  - w$ is not metrizable, in particular to the case where $X$ is the space $L^1$. 
%However $C([0,T]; X_w)$ is a locally convex topological vector space. 

%
\begin{Corollary}
\label{th-compact-L1}
A sequence  $(f_n )_n $ in $C([0,T];L^1)$   is relatively sequentially compact in
 $C([0,T]; L^1 - w)$  if and only if the two following assertions hold true: 
\begin{itemize}
\item the sequence $(f_n )_n $ is weakly equicontinuous on $[0,T]$, 
\item there exists a dense subset $D$ in $[0,T]$ such that for any $t$ in $D$, 
$(f_n (t) )_n$  is weakly relatively compact in $L^1$. 
\end{itemize}
\end{Corollary}
%

%%%%%%%%%%%%%%%%%%%%
\subsection{Existence of compatible data for the viscous model}
\label{sec-compa}

The following result states that the input of the problem  \eqref{vort-ss2} can be approximated by a family of inputs,  for $\nu $ in $(0,1)$, which satisfy the  compatibility conditions for the problems  \eqref{equ:app:ss}. 

\begin{Lemma}
\label{reg-data}
Let $p$ in $ [ 1,+\infty]$ and $ q \in (1,+\infty)$.  
Let  $(\omega^{in} , \omega^{+}  )$ a (\ref{CIV})  in $L^p$.
Let also 
\begin{gather}
\label{hyp:1:DEC}
\text{either } v \in L^{1}_{loc}(\bbR^{+}; W^{1,q}(\calF)) \text{ and (\ref{SSC:v}),}  
\\
\label{hyp:2:DEC}
\text{or } g \text{ in } L^{1}_{loc}(\bbR^{+};W^{1-1/q,q}(\partial \calF)) \text{ and (\ref{SSC:g}).}
\end{gather}
Then there exist some families $ \omega^{in}_{\nu} $, $ \omega_{\nu}^{+} $ and or $ v_{\nu} $ or $ g_{\nu} $ if we assume respectively or \eqref{hyp:1:DEC} or \eqref{hyp:2:DEC}, such that for each  $\nu $ in $(0,1)$, 
 $ \omega^{in}_{\nu} $ is in $ C^{\infty}_{c}(\calF)$, $ \omega_{\nu}^{+} $ is in $C^{\infty}_{c}((0,+\infty) \times \partial \calF)$ and or $ v_{\nu} $ is in $ C^{\infty}(\bbR^+\times \calF ) $ and (\ref{SSC:v}) or
 $ g_{\nu} $ is in $ C^{\infty}((0, +\infty)\times \partial \calF)$ and (\ref{SSC:g}) and such that for $ p < \infty $
\begin{gather} 
    \omega^{in}_{\nu} \longrightarrow \omega^{in}  \quad \text{ in }  L^{p}(\calF) 
    \quad \text{ and } \quad 
    (g_{\nu})^{1/p} \omega^+_{\nu} \cv g^{1/p} \omega^{+}  \text{ in } L^{p}((0,T)\times \partial \mathcal{F}^{+}), \nonumber \\
\text{either } v_{\nu} \longrightarrow v \quad \text{ in } L^{1}_{loc}(\bbR^{+}; W^{1,q}(\calF)) \text{ if we assume \eqref{hyp:1:DEC}},
\label{cv-data} \\
\text{or } g_{\nu} \longrightarrow g \quad \text{ in } L^{1}_{loc}(\bbR^{+};W^{1-1/q,q}(\partial \calF)) \text{ if we assume \eqref{hyp:2:DEC}},  \nonumber
\end{gather}
as $ \nu $ tends to $ 0 $.
Moreover 
\begin{gather}
\label{conv-visc-appr}
\|\omega^{in}_{\nu} \|^2_{L^{2}}+\| g_{\nu}^{1/2}\omega_{\nu}^{+}\|^2_{L^2((0,T)\times \partial \calF^{+})} \leq \frac{C_T}{\sqrt{\nu}} \left( \|\omega^{in}_{\nu} \|_{L^{p}}+\| g_{\nu}^{1/p}\omega_{\nu}^{+}\|_{L^p((0,T)\times \partial \calF^{+})} \right)^2.
\end{gather}
Finally when $ p = \infty $ the above convergence holds for any $ l < p = \infty $.
\end{Lemma}

\begin{proof}

Let first consider the case where we assume \eqref{hyp:2:DEC}.  We use the density of $ C^{\infty}_{c} $ in $ L^p $ for $ p < \infty $ to show the existence of 
\begin{gather*}   
\omega_{\nu}^{in} \to \omega \quad \text{ in } L^{p}(\calF) \quad \text{ and } \quad \frak{f}_{\nu} \to g^{1/p} \omega^{+}  \quad \text{ in } L^{p}((0,T) \times \partial \calF),
\end{gather*}
in particular we choose the approximations such that \eqref{conv-visc-appr} holds. 
Moreover we define $ g_{\nu} $ by convoluting $ g $ with an appropriate positive smoothing kernel of integral $ 1 $ such that $ g_{\nu} $ is (\ref{SSC:g}). We define 
 $\omega_{\nu}^{+} := \mathfrak{f}_{\nu} g_{\nu}^{- 1/p}.$

In the case we assume \eqref{hyp:1:DEC}, we use the same construction as in the previous step with $ g = v\cdot n $. Finally we define $ v_{\nu} $ as the unique smooth solution of 
\begin{equation*}
\begin{cases}
\div v_{\nu} =  0 & \text{ in } \calF,\\
\curl v_{\nu} = \eta_{\nu} \star \curl v & \text{ in } \calF, \\
v_{\nu} \cdot n = g_{\nu} & \text{ in } \partial \calF, \\
\int_{\partial \calS_i} v_{\nu}\cdot \tau = \tilde{\eta}_{\nu}\star \int_{\partial \calS_i} v \cdot \tau  & \text{ for } i \in \calI, 
\end{cases}
\end{equation*}
where $ \eta_{\nu} $ and $ \tilde{\eta}_{\nu}  $ are respectively a spacial and time smoothing convolution kernel.

\end{proof}

%%%%%%%%%%%%%%%%%%%%
\subsection{Existence of smooth solutions to the viscous model}
\label{sec-vcv-exis}

This subsection is devoted to  the existence  of smooth solutions to the 
 the problems  \eqref{vort-ss2}. 
\begin{Lemma}
\label{visc-ex}
Let $ \omega_{\nu}^{in} \in C^{\infty}_{c}(\calF)  $, let $ \omega^{+}_{\nu} \in C^{\infty}_c((0,+\infty) \times \partial \calF) $ and let 
\begin{gather}
\label{hyp:1:DEC:1}
\text{either } v_{\nu} \in C^{\infty}_{c}(\bbR^{+} \times \calF) \text{ and (\ref{SSC:v}),}  \\
\label{hyp:2:DEC:1}
\text{or } g_{\nu} \in  C^{\infty}_{c}(\bbR^+ \times \partial \calF)) \text{ and (\ref{SSC:g}).}
\end{gather}
Then there exists a global unique smooth solution $ \omega_{\nu} $ of the system \eqref{equ:app:ss} associated with the data $ \omega^{in}_{\nu} $, $ \omega_{\nu}^{+} $ and $ v_{\nu} $  if we assume \eqref{hyp:1:DEC:1} and of the system \eqref{equ:app:ss} together with \eqref{u:dec:ss} associated with the data $ \omega^{in}_{\nu} $, $ \omega_{\nu}^{+} $ and $ g_{\nu} $  if we assume \eqref{hyp:2:DEC:1}.
\end{Lemma}

This result is part of the mathematical folklore on boundary values problem for parabolic equation, see for instance \cite[Chapters 8 and 10]{KL} and \cite[Chapter 4 and 7]{LSU}. 
However, since the boundary conditions are quite unusual and for sake of clarity, we give a sketch of proof. 

\begin{proof}

Let us start by dealing with the case where the hypothesis \eqref{hyp:1:DEC:1} holds true. Existence of weak solutions was shown in \cite{Boyer} via a Galerkin method. It is then enough to show \textit{a priori} estimates for higher derivatives. Without loss of generality we assume that $ \nu =1 $ and we do not write the index of the approximation.
We prove by induction on $ n \in \N$ that 
\begin{equation}
  \label{induc}
 \omega \in \bigcap_{i = 0}^n H^{i}(0,T;H^{2n-2i}(\calF)) .
\end{equation}
 This implies that $ \omega $ is smooth. The case $n=0$ was proved in  \cite{Boyer}, it relies on an energy estimate, which is obtained by testing \eqref{weak:for:SS:MB} with the solution $ \omega $ and by some integration by parts, so that we arrive at 
\begin{align}
  \label{induc0}
\frac{1}{2} \int_{\calF}|\omega|^2(t,.) +  \int_{0}^{t}\int_{\calF^{-}} g\frac{|\omega|^2}{2} + \int_{0}^{t}\int_{\calF} |\nabla \omega |^2 \leq \int_{\calF} |\omega^{in}|^2  + \int_{0}^{t} \int_{\partial \calF^{+}}
(-g) \frac{|\omega^+|^2}{2}.
\end{align}
   Before to move on the iteration step, we first tackle the case where $ n= 1 $ to display the method. 

\paragraph*{Proof of \eqref{induc}  in the case where $ n= 1 $.}
 We test the weak formulation \eqref{weak:for:SS:MB} with $ \partial_t \omega $ and after some integrations by parts, we have
\begin{align}
  \label{induc1}
\frac{1}{2}\int_{\calF}| \nabla \omega|^2&(t,.) +  \frac{1}{4}\int_{0}^{t} \int_{\calF} |\partial_t \omega|^2 + \frac{1}{8}\int_{0}^{t}\int_{\calF} |\Delta \omega|^2 \leq \frac{3}{4} \int _0^t\int_{\calF}\left| v\cdot \nabla \omega \right|^2 + \int_{0}^t \int_{ \partial \calF_{+}} (\omega-\omega^{+})g\partial_{t}\omega  .
% \\ \leq  \, & \frac{3}{4} \int_0^t \left(\|v\|^2_{L^{\infty}(\calF)}\int_{\calF}|\nabla \omega|^2\right) - \frac{1}{2}\int_{\partial \calF^{+}} g|\omega- \omega^{+}|^2(t,.) + \frac{1}{2}\int_{\partial \calF^{+}}g|\omega^{in}-\omega^{+}|^2  ,
% \\& \, + \int_0^{t} \int_{\partial \calF^{+}} (-g) |\omega-\omega^{+}|^2 + \int_0^t \int_{\partial \calF^{+}} \frac{|\partial_t g|^2}{2|g|} + \frac{1}{2}\int_0^t \int_{\partial \calF^{+}} |g||\partial_t \omega^{+}|^2 + \frac{1}{2}\int_{\calF} |\nabla \omega^{in}|^2.
\end{align}
This leads to
\begin{align*}
\frac{1}{2}\int_{\calF}|\nabla \omega |^2 (t,.) + & \frac{1}{4}\int_{0}^{t} \int_{\calF} |\partial_t \omega|^2 + \frac{1}{8}\int_{0}^{t}\int_{\calF} |\Delta \omega|^2 + \frac{1}{2}\int_{\partial \calF^{+}} (-g)|\omega- \omega^{+}|^2(t,.) \\
\leq & \, \frac{3}{4} \int_0^t \left(\|v\|_{L^{\infty}(\calF)}\int_{\calF}|\nabla \omega|^2\right) - \int_0^t \int_{\partial \calF^{+}}g|\omega^{in}-\omega^{+}|^2 \\ & \, + \int_0^t \int_{\partial \calF^{+}} \frac{|\partial_t g|^2}{2|g|} + \frac{1}{2} \int_0^t \int_{\partial \calF^{+}} |g|\partial_t \omega^{+}|^2 + \frac{1}{2}\int_{\calF} |\nabla \omega^{in}|^2.
\end{align*}
All the terms in the right hand side above only depend on the data, except the first one, which can be handled by a Gronwall lemma. 
Using the classical elliptic estimate: 
%The above estimates together with  
%
\begin{equation*}
\int_{\calF} |\nabla^2 \omega|^2 \leq C\int_{\calF} \big( |\Delta \omega |^2  + |\nabla \omega |^2 + |\omega |^2 \big),
\end{equation*}
%
% this follow for regularity of the elliptic problem $ -\Delta v + v = - \Delta \omega + \omega $ with boundary conditions $ \nabla v \cdot n = (\omega-\omega^{+})g $. 
%
we deduce  that   \eqref{induc} holds true for  $ n= 1 $.
%$ \omega \in \bigcap_{i=0}^{1} H^{i}(0,T;H^{2-2i}) $.

\paragraph*{Inductive step.} We suppose that  \eqref{induc}  holds true up to some order $n$ and we are going to prove that it also holds true at the order $n+1$. 
To do that we apply $ \partial_{t}^n $ to  \eqref{weak:for:SS:MB}. 
Then $ \tilde{\omega} := \partial_t \omega^n $ satisfies
\begin{align}
\label{preums}
\partial_t \tilde{\omega} + v\cdot \nabla \tilde{\omega} = &\,  \nu \Delta \tilde{\omega} 
+  \, \sum_{i = 1}^{n} \partial_t^i v\cdot \nabla \partial_t^{n-i}\omega  ,\quad && \text{ in  } \bbR^{+} \times \calF,  \\
\partial_n \tilde{\omega} = & \, (\tilde{\omega}-\tilde{\omega}^+)g + \sum_{i = 1}^{n}\partial_t^{n-i}(\omega-\omega^{+})\partial_t^i g ,\quad && \text{ in } \bbR^{+}\times \partial \calF.
\end{align}
For this system we have 
an energy estimate similar to \eqref{induc0}, obtained by multiplying \eqref{preums} by $ \tilde{\omega}$,  except that two new terms appear due to the last terms in the right hand side above. Proceeding as before, we have only to estimates the two new terms
\begin{gather*} 
I = \int_{0}^{t} \int_{\calF} \sum_{i = 1}^{n} \partial_t^i v\cdot \nabla ( \partial_t^{n-i}\omega) \tilde{\omega} \quad\text{ and } 
 \quad II = \int_{0}^{t} \int_{\partial \calF^+ } \sum_{i = 1}^{n}\partial_t^{n-i}(\omega-\omega^{+})(\partial_t^i g) \tilde{\omega}.
\end{gather*}
For the first term, it holds
\begin{align*}
|I| \leq & \, \int_0^{t}\int_{\calF} |\tilde{\omega}|^2 + \sum_{i=1}^{n}\int_0^{t}\int_{\calF} |\partial_t^i v \cdot \nabla \partial_t^{n-1} \omega|^2 \\
\leq & \,  \int_0^{t}\int_{\calF} |\tilde{\omega}|^2 + \sum_{i=1}^{n}\int_0^{t}\|\partial_t^i v \|_{L^{\infty}(\calF)}^2\int_{\calF} | \nabla \partial_t^{n-1} \omega|^2 \\
\leq & \, \int_0^{t}\int_{\calF} |\tilde{\omega}|^2 + C,
\end{align*}
where $ C $ depends from the data, thanks to the previous steps. The term $II$ can be handled 
similarly. 
%
%\begin{align*}
%|II| \leq & \,  \frac{1}{2} \int_0^t \int_{\calF} \left(|\tilde{\omega}|^2+|\nabla \tilde{\omega}|^2 \right)+C \sum_{i=1}^{n} \int_{0}^{t}\int_{\partial \calF^{+}}|\partial_t^i g|^2|\partial_t^{n-i}(\omega-\omega^{+})|^2  \\ \leq & \,  \frac{1}{2} \int_0^t \int_{\calF} \left(|\tilde{\omega}|^2+|\nabla \tilde{\omega}|^2 \right)+C,\end{align*}
%where $ C $ depends from the data.

Let now  multiply  the equation  \eqref{preums} by  $ \partial_t \tilde{\omega} $ and integrate over $(0,t) \times \calF$. This provides an estimate similar to   \eqref{induc1} with two 
extra  terms  due to the last terms in the right hand side of \eqref{preums}. These two terms are: 
\begin{gather*} 
I^{im} = \int_{0}^{t} \int_{\calF} \sum_{i = 1}^{n} \partial_t^i v\cdot \nabla \partial_t^{n-i}\omega \partial_t \tilde{\omega} \quad\text{ and }  \quad II^{im} = \int_{0}^{t} \int_{\partial \calF^+ } \sum_{i = 1}^{n}\partial_t^{n-i}(\omega-\omega^{+})\partial_t^i g \partial_t \tilde{\omega}.
\end{gather*}
As before we have
\begin{align*}
|I^{im}| \leq & \, \frac{1}{8} \int_0^{t}\int_{\calF} |\partial_t \tilde{\omega}|^2 + C.
\end{align*}
The second one is more technical. We rewrite
\begin{align*}
II^{im} = & \, \sum_{i = 1}^{n} \int_{\partial \calF^+ } \partial_t^{n-i}(\omega(t,.)-\omega^{+}(t,.))\partial_t^i g(t,.) \tilde{\omega}(t,.)- \sum_{i = 1}^{n}  \int_{\partial \calF^+ } \partial_t^{n-i}(\omega^{in}-\omega^{+}(0,.))\partial_t^i g(0,.) \tilde{\omega}(0,.) \\
 & \,  - \sum_{i = 1}^{n} \int_{0}^{t} \int_{\partial \calF^+ } \partial_t^{n-i+1}(\omega-\omega^{+})\partial_t^i g \tilde{\omega} - \sum_{i = 1}^{n} \int_{0}^{t} \int_{\partial \calF^+ } \partial_t^{n-i}(\omega-\omega^{+})\partial_t^{i+1} g \tilde{\omega} \\
= & \, \sum_{i = 1}^{n}\int_{\partial \calF^+ } \partial_t^{n-i} (\omega(t,.)-\omega^{+}(t,.))\partial_t^i g(t,.) \tilde{\omega}(t,.)- \sum_{i = 1}^{n}  \int_{\partial \calF^+ } \partial_t^{n-i}(\omega^{in}-\omega^{+}(0,.))\partial_t^i g(0,.) \tilde{\omega}(0,.) \\
 & \,  - \int_{0}^{t} \int_{\partial \calF^+ } \partial_t^{n}(\omega-\omega^{+})\partial_t^i g \tilde{\omega} - \sum_{i = 2}^{n} \int_{0}^{t} \int_{\partial \calF^+ } \partial_t^{n-i+1}(\omega-\omega^{+})\partial_t^i g \tilde{\omega} \\ 
 & \, - \sum_{i = 1}^{n} \int_{0}^{t} \int_{\partial \calF^+ } \partial_t^{n-i}(\omega-\omega^{+})\partial_t^{i+1} g \tilde{\omega}  \\
=  & \, \sum_{i = 1}^{n} \int_{\partial \calF^+ } \partial_t^{n-i}(\omega(t,.)-\omega^{+}(t,.))\partial_t^i g(t,.) \tilde{\omega}(t,.)- \sum_{i = 1}^{n}  \int_{\partial \calF^+ } \partial_t^{n-i}(\omega^{in}-\omega^{+}(0,.))\partial_t^i g(0,.) \tilde{\omega}(0,.) \\
 & \,  - \int_{0}^{t} \int_{\partial \calF^+ } (\tilde{\omega}-\tilde{\omega}^{+})\partial_t g \tilde{\omega} - \sum_{i = 2}^{n} \int_{0}^{t} \int_{\partial \calF^+ } \partial_t^{n-i+1}(\omega-\omega^{+})\partial_t^i g \tilde{\omega} \\ 
 & \, - \sum_{i = 1}^{n} \int_{0}^{t} \int_{\partial \calF^+ } \partial_t^{n-i}(\omega-\omega^{+})\partial_t^{i+1} g \tilde{\omega}.
\end{align*} 
Using the above equality we deduce that
\begin{align*}
| II^{im}| \leq & \, \frac{1}{4}\|\tilde{\omega}\|^2_{H^1(\calF)} + \frac{1}{4} \| \tilde{\omega}^{in}\|_{H^1(\calF)}^2 + \sum_{i=1}^{n}\|\partial_t^{i}g\|_{C^0}\|\partial_t^{n-i}(\omega-\omega^{+})\|_{C^{0}(0,T;H^{1}(\calF))}^2 \\
& \, +  \|\partial_t g \|_{C^0(0,t)}  \int_0^{t}\|\tilde{\omega}\|_{H^{1}(\calF)}^2 + \frac{3}{2}   \int_0^{t}\|\tilde{\omega}\|_{H^{1}(\calF)}^2   + \int_0^t \int_{\partial \calF^+}|\tilde{\omega}^+ \partial_t g |^2  \\ & \, + \sum_{i=2}^{n}\int_{0}^{t} \int_{\partial \calF^+ }   |\partial_t^{n-i+1}(\omega-\omega^{+})\partial_t^i g|^2 +  \sum_{i=2}^{n}\int_{0}^{t} \int_{\partial \calF^+ }   |\partial_t^{n-i}(\omega-\omega^{+})\partial_t^{i+1} g|^2.
\end{align*}
The first term of the right hand side can be absorbed by the left hand side, the first two terms of the second line are tackled by the Gr\"onwall argument,  and all the remaining ones are bounded.

We deduce that $ \tilde{\omega} = \partial_t^n \omega $ is \textit{a priori} bounded in $ H^{1}(0,T;L^{2}(\calF)) \cap L^{2}(0,T;H^2(\calF)) $. The desired regularity follows by using the equation \eqref{weak:for:SS:MB}.   This allows to conclude  that \eqref{induc} is true for every $ n $ in $\N$. This achieves  the proof of Lemma 
\ref{visc-ex} in  the case where the hypothesis \eqref{hyp:1:DEC:1} holds true.
 
\bigskip 
 
To prove  existence  in the case where hypothesis \eqref{hyp:2:DEC:1} holds, we are going to use  the Schauder fixed point theorem which asserts that if  $\mathcal{B}$ is a nonempty convex closed subset of a 
normed space $\mathcal{X}$ and $F:\mathcal{B}\mapsto\mathcal{B}$ is a continuous mapping 
 such that $ F(\mathcal{B})$ is contained in a compact subset of $ \mathcal{B}$, then
 $F$  has a fixed point.

%%%
\paragraph{Definition of an appropriate operator.}
 Let $n$ in $\N$,  $R >0$ and 
\begin{gather*}
\mathcal{Z} = \left\{ \omega \in \bigcap_{i=0}^{n} H^{i}(0,T;H^{2n-2i}(\calF)) \quad \text{ such that } \quad 
\| \omega \|_Z := \| \omega \|_{\bigcap_{i=0}^{n} H^{i}(0,T;H^{2n-2i}(\calF))} \leq R \right\} . 
\end{gather*}
Define the map $ F:\mathcal{Z} \longrightarrow \bigcap_{i=0}^{n} H^{i}(0,T;H^{2n-2i}(\calF)) $ where $ F(\omega) = \bar{\omega} $ is the solution of 
\begin{align*}
\partial_t \bar{\omega} + v_{\omega}\cdot \nabla \bar{\omega}- \Delta \bar{\omega} = & \,  0  \quad && \text{ in  } \bbR^{+} \times \calF, \\
\partial_n \bar{\omega} =  & \, (\bar{\omega}-\omega^+)g \quad && \text{ in } \bbR^{+}\times \partial \calF \\
\text{div } v_{\omega} =  &\, 0  \quad && \text{ in  } \bbR^{+} \times \calF,  \\
\text{curl } v_{\omega} = &\, \omega   \quad && \text{ in  } \bbR^{+} \times \calF,  \\
v_{\omega} \cdot n = & \, g   \quad && \text{ in } \bbR^{+}\times \partial \calF,  \\
\oint_{\partial \calS^{i}} v_{\omega} \cdot  \tau =  &\, \calC_i^{in}+\int_0^t \oint_{\partial \calS^{i}} \omega g    \quad && \text{ for any  }  i \in I_- , \\
\oint_{\partial \calS^{i}} v_{\omega} \cdot  \tau = &\, \calC_i^{in}+\int_0^t \oint_{\partial \calS^{i}} \omega^{+} g   \quad && \text{ for any  }  i \in I_+ .
\end{align*}
For any $\omega$ in $Z$, we have that $ v_{\omega} $ is in $\bigcap_{i=0}^{n}H^i(0,T;H^{2n-2i+1}(\calF))$. 
Then, by using the  \textit{a priori} estimates  above, we observe that the $ \mathcal{Z} $-norm of $ \bar{\omega} $ depends on the $ \bigcap_{i=0}^{n-1}H^i(0,T;H^{2n-2i}(\calF)) $ of $ v_{\omega} $, which converges,  as $ t $ tend to zero, to zero uniformly with respect to $ \omega \in \mathcal{Z} $.
 Therefore for  $ T $ small  enough, we conclude that $ F(\mathcal{Z}) \subset \mathcal{Z} $.

Let us now prove that  $ F $ is relatively compact.  
Let $ (\omega_{j} )_j$ a  bounded sequence in $ \mathcal{Z} $. Then, up to a subsequence, $ \omega_{j} \cv \omega $ in $ \mathcal{Z} $. By Rellich's theorem the convergence is  strong in $$ \bigcap_{i=0}^{n-1} H^{i}(0,T; H^{2n-2i-1 }(\calF)) .$$
 We deduce that the corresponding velocity $ v_{j} = v_{\omega_{j}} $ converges to $v $ in 
$ \bigcap_{i=0}^{n-1} H^{i}(0,T; H^{2n-2i }(\calF)) .$
Moreover for any $j$, the function $ w_j = \bar{\omega}-\bar{\omega}_{j}$ satisfies the system 
\begin{align*}
\partial_{t} w_j + v\cdot \nabla w_j - \Delta w_j = & \, -(v-v_{j})\cdot \nabla \bar{\omega}_{j}  \quad && \text{ in  } \bbR^{+} \times \calF,\\
\partial_n w_j = & \, w_j g   \quad &&  \text{ in } \bbR^{+}\times \partial \calF,
\end{align*} 
with zero initial data. 
We observe that 
\begin{align*}
%\sum_{i=0}^{n-1} \left\| \partial_t^i((v-v_j )\cdot \nabla \bar{\omega}_{j}\right\|_{L^2(0,T;L^2(\calF))} 
\left\| (v-v_j)\cdot \nabla \bar{\omega}_j \right\|_{H^{n-1}(0,T;L^2(\calF))} 
\leq & \, \sum_{i=0}^{n-2} \left\| \partial_t^i(v-v_j )\|_{L^{ \infty}(0,T;L^{\infty}(\calF))} \| \partial_{t}^{n-1-i}\nabla \bar{\omega}_{j}\right\|_{L^2(0,T;L^2(\calF))} \\ & \, +\left\| \partial_t^{n-1}(v-v_j )\|_{L^{2}(0,T;L^{\infty}(\calF))} \| \partial_t^{n-1-i}\nabla \bar{\omega}_{j}\right\|_{L^{\infty}(0,T;L^2(\calF))} \\
& \, \longrightarrow 0.
\end{align*}
Proceeding in the same way for higher derivatives, we obtain that 
\begin{gather*}
\sum_{i = 0}^{n-1}\left\| (v-v_j)\cdot \nabla \bar{\omega}_j \right\|_{H^i(0,T;H^{2n-2-2i}(\calF))} \longrightarrow 0.
\end{gather*}
Then by using the \textit{a priori} estimates above, we deduce that $ w_j $ converges to $ 0 $ in $ \mathcal{Z} $. Thus $ F $ is relatively compact. 
The continuity of $F$ can be proved along the same lines. 
Thus Schauder's fixed point theorem can be applied. It implies that $F$ has a fixed point in $\mathcal{Z}$. This has proved the local in time existence of smooth solutions.
Moreover  the existence to all $  [0,T] $  can be deduced from the \textit{a priori} estimates.  In fact if we suppose by contradiction that there exists a maximal time of existence $ t_{*} < T $, the  \textit{a priori}  estimates ensure that $ \omega(t_{*}) $ is enough regular to apply again the local existence result and we obtain a contradiction. Uniqueness follows from the energy estimate \eqref{induc0} and Gr\"onwall's lemma. 
\end{proof}

%%%%%%%%%%%%%%%%%%%%
\subsection{Convergence of the approximations}
\label{sec-vcv-cv}

In this subsection we establish the convergence of the solutions to 
 the problems  \eqref{vort-ss2} in the vanishing viscosity limit.

\begin{Proposition}
\label{visc-appr}
Let $p$ in $[1,+\infty]$. Let $ q  = p $ in the case where $ p > 1 $ and $ q > 2 $ in the case where $ p = 1$. 
Let $ g $ in $ L^{1}_{loc}(\bbR^{+};W^{1-1/q,q}(\calF))$ and (\ref{SSC:g}). 
Let  $(\omega^{in} , \omega^{+}  )$ a (\ref{CIV})  in $L^p$. 
Let some families $ \omega^{in}_{\nu} $, $ \omega_{\nu}^{+} $ and $ g_{\nu} $ for $\nu $ in $(0,1)$, as in 
Lemma \ref{reg-data}. 
Let  $ \omega_{\nu} $ the corresponding global unique smooth solution of the system \eqref{equ:app:ss}-\eqref{u:dec:ss} associated with the data $ \omega^{in}_{\nu} $, $ \omega_{\nu}^{+} $ and $ g_{\nu} $, from Lemma \ref{visc-ex}. 
Then there are a subsequence of $ \omega_{\nu} $  and a subsequence $ v_{\nu} $
which we still denote  $ \omega_{\nu} $ and $ v_{\nu }$  and which satisfy 
\begin{gather} \label{cv-fc}
    \omega_{\nu} \longrightarrow \omega \quad \text{ in } C_{w}([0,T];L^{p}(\calF)) 
    \quad 
    g_{\nu}^{1/p} \omega_{\nu} \cv g^{1/p} \omega^{-}  \text{ in } L^{p}((0,T)\times \partial \mathcal{F}^{-}), \\ \label{cv-fc2}
\text{ and } \quad  \sqrt{\nu}  \nabla \omega_{\nu}   \quad \text{bounded in } L^{2}([0,T];L^2(\calF)) ,\\ \label{cv-fc3}
 v_{\nu}-v_{g_{\nu}} \longrightarrow v -v_{g} \quad \text{ in } C_{w}([0,T];W^{1,p}(\calF)),  
\end{gather}
in the case where  $ p < \infty $ and which satisfy  \eqref{cv-fc} with any real number greater than $1$ instead of $p$ in the case where $p=+\infty$. 
Moreover $v$ and $\omega$ satisfy \eqref{u:dec:ss} and \eqref{vort-ss2-e}.

\end{Proposition}
\begin{proof}[ Proof of Proposition \ref{visc-appr}]

We will proceed in six steps.

  \paragraph{Step 1.  \textit{A priori} bounds.}
  Let $ G $ a positive even convex function. 
Formally, multiplying  \eqref{equ:app:ss1}  by   $ G'(|\omega_{\nu}|) \frac{ \omega_{\nu} }{|\omega_{\nu}| } $ and using the rule
$$ \partial_i \vert \omega_{\nu}\vert  = \frac{ \omega_{\nu} }{|\omega_{\nu}| }  \partial_i \omega_{\nu} ,$$
for $i=t$, $x_1$ or $x_2$, we arrive at 
\begin{equation} \label{riri}
\partial_t G(\omega_{\nu}) + v_{\nu}\cdot \nabla G(\omega_{\nu})  =   \nu \Delta G(\omega_{\nu} ) -  \nu G''(\omega_{\nu} ) \vert  \nabla \omega_{\nu} \vert^2 . 
\end{equation}
Integrating on $\calF$ and using the boundary conditions \eqref{equ:app:ssin}  we have that 
$$
\int_{ \calF} v_{\nu}\cdot \nabla G(\omega_{\nu})  \, dx =
\int_{ \calF} \div ( G(\omega_{\nu}) v_{\nu} )  \, dx =
\int_{\partial \calF}  G(\omega_{\nu}) g_{\nu} \, ds , 
$$
and that 
$$
\nu \int_{ \calF}   \Delta G(\omega_{\nu})  \, dx =
\nu  \int_{ \partial\calF} \partial_n  G(\omega_{\nu}) \, dx =
\int_{\partial \calF_+}  \frac{ \omega_{\nu} }{|\omega_{\nu}| }   (\omega_{\nu} -\omega^{+}_{\nu})g_{\nu} \, ds .
$$
Thus, by integrating  \eqref{riri} on $[0,T] \times \calF$, 
we obtain:
\begin{align}
\label{est:con:Gun-pre}
\int_{\calF} G(|\omega_{\nu}(t,\cdot)|) + \int_0^{t} \int_{\partial \calF^{-}} g_{\nu} G(|\omega_{\nu}|) & + \nu \int_0^t \int_{\calF} G''(|\omega_{\nu}| ) \vert  \nabla \omega_{\nu} \vert^2 \leq  \int_{\calF} G(|\omega_{\nu}^{in}|) 
\\ \nonumber  &  \quad \quad 
   + \int_0^{t} \int_{\partial \calF^{+}} (- g_{\nu}) \Big( G(|\omega_{\nu}|) - \frac{ \omega_{\nu} }{|\omega_{\nu}| }   ( \omega^{+}_{\nu}- \omega_{\nu} ) \Big).
\end{align}
Using the convexity of the function $G$, we have that 
$$ G(\vert x\vert) - G'(\vert x\vert) \frac{  x}{ \vert x\vert} (y-x) - G(\vert y\vert) \leq 0 \quad \text{ and } \quad G''(|x|) \geq 0.$$
We use these inequalities with $x = \omega_{\nu}$ and $y = \omega^{+}_{\nu}$ to bound the two last terms of \eqref{est:con:Gun-pre} and the last term of the right hand side and we arrive at: 
\begin{equation}
\label{est:con:Gun}
\int_{\calF} G(|\omega_{\nu}(t,\cdot)|) + \int_0^{t} \int_{\partial \calF^{-}} g_{\nu} G(|\omega_{\nu}|)
 \leq  \int_{\calF} G(|\omega_{\nu}^{in}|) + \int_0^{t} \int_{\partial \calF^{+}} (-g_{\nu}) G(|\omega_{\nu}^+|).
\end{equation}
This bound can be rigorously justified by considering $\sqrt{x^2+\eps}$ as an appropriate sequence of regularizations of the absolute value and by  passing to the limit as $\eps$ goes to $0$.
Using the peculiar cases of the power functions we have that $\omega_{\nu}$ satisfies the 
 \textit{a priori}   bounds \eqref{apq} and \eqref{apinf}. 
 
\paragraph{Step 2. Convergence of the vorticity in the case where $ p > 1 $.} In the case where $ p > 1 $, we use the  \textit{a priori} bound \eqref{est:con:Gun} in the case where $G(x) = x^p$.
 This implies that $ \omega_{\nu} $ converges to $\omega $ in $ L^\infty ([0,T];L^{p}(\calF)-w)$
  and $g_{\nu}^{1/p} \omega_{\nu}$ converges to $ \mathfrak{f} $ in $ L^{p}((0,T)\times \partial \mathcal{F}^{-}) $. The second convergence in \eqref{cv-fc} follows from the fact that $ g > 0 $ on $ \partial \calF^{-} $ and defining $ \omega^{-} =  \mathfrak{f} / g^{1/p} $. 
To reinforce the convergence of  $ \omega_{\nu}$ as a strong convergence in time such as stated in the first  convergence in \eqref{cv-fc}, we are going to establish some bounds on the time derivative. 

In the case where $ p \geq 4/3 $ it follows from \eqref{u:dec:ss}, from \eqref{CZ} and from the Sobolev embedding theorem 
that the sequence $ (v_{\nu} \omega_{\nu} )_\nu$ is uniformly bounded in $ L^1 ((0,T); L^q(\calF) )$ with $q= 2p/ (2-p)$, so that by H\"older's inequality, 
the term $ v_{\nu} \omega_{\nu} $ is uniformly bounded in $ L^1((0,T) \times \calF )$. Then it follows from \eqref{equ:app:ss1} and from the Sobolev embedding theorem  that the sequence 
$ (\partial_t \omega_\nu )_\nu $ is bounded in a Sobolev space of negative order. 

In the case where $ p \in (1,4/3] $ is more tricky. For each $\nu$ in $(0,1)$, the smooth solution $ \omega_{\nu} $ of the system \eqref{equ:app:ss} also satisfies the weak formulation \eqref{weak:for:SS:MB}, and by 
using  \eqref{lasalle} and \eqref{u:dec:ss}, we obtain that it also satisfies the following viscous weak symmetrized formulation: 
 for any $ \varphi \in C^{\infty}_{c}(\bbR^{+}\times \calF) $, 
\begin{align}
\label{wf:1:equ:ss-nu}
\int_{\calF} \omega_{\nu}^{in}\varphi(0,.) dx & + \int_{\bbR^{+}}  \int_{\calF} \omega_{\nu} \big( \partial_t \varphi + v_g \cdot \nabla \varphi \big) \, dx dt
 +  \sum_{i}  \int_{\bbR^+}  \calC_{i,\nu}(t)  \int_{\calF} \omega_{\nu} X_i \cdot \nabla \varphi  \, dx \, dt \\ & 
 +  \int_{\bbR^+}\int_{\calF} \int_{\calF}H_{\varphi}(x,y)\omega_{\nu} (t,x)\omega_{\nu}(t,y) \, dx \, dy \, dt
 - \nu \int_{\bbR^{+}}\int_{\calF} \nabla \omega_{\nu} \cdot \nabla \varphi {\color{red} \, = 0,  } \nonumber 
\end{align}
where
\begin{equation}
\label{formuCi-nu}
\calC_{i,\nu}(t) = \calC_i^{in} - \int_{0}^{t} \int_{\pS^{i}} \omega_{\nu}^\pm g_{\nu} \ \text{ for } i \in \calI^\pm .
\end{equation}
 Then  we deduce the uniform bound for $ \partial_t \omega_{\nu}  $  thanks to Lemma  \ref{Hbd},
and
\begin{align*}
\int_{\calF} v_{g_{\nu}}\cdot \nabla \varphi \omega_{\nu} = \int_{\calF} v_{g_{\nu}} \cdot \nabla \varphi \Delta \eta_{\nu} = \int_{\calF} \Delta v_{g_{\nu}} \cdot \nabla \varphi \eta_{\nu} + \int_{\calF} \nabla v_{g_{\nu}}: \nabla^2 \varphi \eta_{\nu} - \int_{\calF} v_{g_{\nu}} \otimes \nabla  \eta_{\nu} : \nabla^2 \varphi, 
\end{align*}
where $ \Delta \eta_{\nu} = \omega_{\nu} $ in $ \calF $ and $ \eta_{\nu} = 0 $ in $ \partial \calF $. Moreover recall that $ v_{g_{\nu}} = \nabla \phi_{g_{\nu}} $ solution of $ \Delta \phi_{g_{\nu}} = 0 $ in $ \calF $ and $ \nabla \phi_{g_{\nu}} \cdot n = g_{\nu} $, in particular $ \Delta v_{g} = 0 $. It follows
\begin{align*}
\left| \int_{\calF} v_{g_{\nu}}\cdot \nabla \varphi \omega_{\nu} \right| \leq & \, \left| \int_{\calF} \nabla v_{g_{\nu}}: \nabla^2 \varphi \eta_{\nu} \right| + \left| \int_{\calF} v_{g_{\nu}} \otimes \nabla  \eta_{\nu} : \nabla^2 \varphi \right| \\ \leq & \, \| \nabla v_{g_{\nu}}\|_{L^p(\calF)} \|\nabla^2 \varphi \|_{L^{\infty}(\calF)}\|\eta_{\nu} \|_{L^{\frac{p}{p-1}}(\calF)} \\ & \, +  \| v_{g_{\nu}}\|_{L^{2}(\calF)}\|\eta_{\nu} \|_{L^{2}(\calF)}\|\nabla^2 \varphi \|_{L^{\infty}(\calF)} \\
\leq & \, \| v_{g_{\nu}}\|_{W^{1,p}(\calF)}\|\omega_{\nu}\|_{L^{p}(\calF)}\|\varphi\|_{W^{2,\infty}_{0}(\calF)},
\end{align*}
where from the Sobolev embedding in dimension two we have $ W^{2,k} \subset L^{\infty} $ and $ W^{1,p} \subset L^2 $ for $ p> 1$.

We now apply the following version of  the Aubin-Lions lemma, see for example \cite[Lemma 11]{BodyDelort},  to  $f_\nu = \omega^\nu$
 with  $X= L^{\frac{ p}{p-1 } } (\calF )$ (respectively $ L^{1 } (\calF )$ if $p=+\infty$) 
 and $Y =H^{M+1}_{0}(\calF )$,   with $M$ large enough.

\begin{Lemma}
\label{weakc}
 Let $X$ and $Y$ be two Banach spaces such that $Y$ is dense in $X$ and $X$ is separable. Assume that $(f_n )_n $ is a bounded sequence in  $L^{\infty}(0,T; X')$ such that  $(\partial_t f_n )_n $ is  bounded in $L^{1}(0,T; Y')$. Then $(f_n )_n $ is relatively compact in  $C( [0,T ]; X' - w* )$.
\end{Lemma}
In particular this proves the first  convergence in \eqref{cv-fc} in the case  where $p>1$.

  \paragraph{Step 3.  Convergence  of the vorticity in the case where $p=1$.  }
  In the case where $p=1$, by \eqref{cv-data} and Dunford-Pettis' theorem we have that   $( \omega^{in}_{\nu})$ and $ ( \omega^+_{\nu} g^+_{\nu})$  are uniformly integrable
  respectively  in    $L^{p}(\calF) $ and  in  $L^{p}((0,T)\times \partial \mathcal{F}^{+})$. 
  Therefore, by Lemma \ref{Delavalle:ain} and  Corollary \ref{Cor:DLVPoussin}, there 
exists an even convex function $ G: \bbR \to \bbR^{+} $ such that
\begin{equation*}
\lim_{|s| \longrightarrow + \infty} \frac{G(s)}{|s|} = + \infty , \quad 
\sup_{\nu} \int G(|\omega^{in}_{\nu} |) < + \infty ,
\\ \quad \text{ and } \quad \sup_{\nu}\int_{\partial \calF^{+}} G(  \omega^+_{\nu}) (-g_{\nu}) < + \infty. 
\end{equation*}  
Then by \eqref{est:con:Gun}, we deduce that 
\begin{equation*}
\sup_t \sup_{\nu} \int G(|\omega_{\nu} (t,\cdot)|) < + \infty ,
\\ \quad \text{ and } \quad \sup_{\nu}\int_{\partial \calF^{-}}G(  \omega_{\nu}) g_{\nu}  < + \infty. 
\end{equation*}  
Therefore, using again  Lemma \ref{Delavalle:ain} and Dunford-Pettis' theorem (more precisely the parts regarding the reverse statements) for any $t$ in $[0,T]$ the sequence $(\omega_{\nu} (t,\cdot))_\nu$  is weakly relatively compact in $L^1$. 
On the other hand we can obtain a uniform bound of $ \partial_t 	\omega_{\nu} $ in a Sobolev space of negative order by proceeding as in the case $ p \in (1,4/3] $ by using that $ v_{g_{\nu}} $ in $ L^{\infty}(\calF) $, a consequence of the regularity of $ g_{\nu} $.
Then it suffices to apply Corollary 
\ref{th-compact-L1} to conclude that the first  convergence in \eqref{cv-fc} holds true in the case  where $p=1$.

\paragraph{Step 4. Endgame.} 
%Bound of the gradient of the vorticity.} 
Since  $ \omega_{\nu} $ and $ v_{\nu }$ are related by \eqref{u:dec:ss}
the convergence of the velocity in 
\eqref{cv-fc3} straightforwardly follows from the convergence of the vorticity, see in particular the property \eqref{CZ}. 
Moreover  we can pass to the limit in the relation  \eqref{u:dec:ss}  and we obtain that $ \omega$ and $ v$ are related by \eqref{u:dec:ss}.  For more details on the convergence of the part of the velocity associated with the circulations see \eqref{T4}.

The bound of the gradient of the vorticity in \eqref{cv-fc2}
follows from the $ L^2 $ a \textit{priori} bounds, corresponding to the case where $G(x) = x^2$ and from the hypothesis \eqref{conv-visc-appr} on the data.
 
\end{proof} 
%%%%%%%%%%%%%%%%%%%%%%

%%%%%%%%%%%%%%%%%%%%
\section{Proof of Theorem \ref{def:wea:sol:ss} on distributional solutions}
\label{cas43}

Let,  for $\nu $ in $(0,1)$,  $ \omega^{in}_{\nu} $, $ \omega_{\nu}^{+} $ and $ g_{\nu} $, as in 
Lemma \ref{reg-data}, and $ \omega_{\nu} $ the corresponding global unique smooth solution of the system \eqref{equ:app:ss} as in Lemma \ref{visc-ex}. 
 These solutions satisfy \eqref{weak:for:SS:MB} for any $ \varphi \in C^{\infty}_{c}(\bbR^{+}\times \overline{\calF}) $. 
 By applying Proposition \ref{visc-appr}, in the case where the family $(v_{\nu})_{\nu \in (0,1)}$ is related to the family of vorticity $(\omega_{\nu})_{\nu \in (0,1)}$ by \eqref{u:dec:ss},  we obtain that, up to a subsequence, 
 $( \omega_{\nu} )_\nu$ and $(g_{\nu}^{1/p} \omega_{\nu})_\nu$ satisfy \eqref{cv-fc}, and $(v_{\nu})_{\nu \in (0,1)}$ satisfies 
  \eqref{hyp-vnu}, and the limit vorticity $\omega$ and the limit velocity $v$  satisfy \eqref{u:dec:ss} and \eqref{vort-ss2-e}.
  
By the  Rellich-Kondrachov theorem,  using that  $2p/(2-p) > p/(p-1)$  when $ p > 4/3 $,  we obtain that,  up to a subsequence, $(v_{\nu})_{\nu \in (0,1)}$ converges, strongly,  in $$L^1_{loc}(\bbR^+ ;L^{p/(p-1)}(\calF)).$$

These convergences allow to pass to the limit in  \eqref{weak:for:SS:MB} and to  arrive at \eqref{wf:equ:ss:noren}. 
Moreover the continuity in time with values in $ L^p $ of $ \omega $ and the equality  \eqref{apq} follow from the fact that $ (\omega, \omega^{-})$ is also a renormalized solution which will be proved in the next section. The bound \eqref{apinf} follows from the lower semi-continuity of the weak limits.

%%%%%%%%%%%%%%%%%%%%
\section{Proof of Theorem \ref{exi:Lp:ss} on renormalized solutions}
\label{sec-pr-ren}

This section is devoted to the proof of Theorem \ref{exi:Lp:ss} on the existence of renormalized solutions to the Euler equations in presence of sources and sinks in the case where
the input vorticities are in $L^p$, with  $ p \in (1,\infty]$. The proof also relies on the viscous approximations built in the previous section. 
The next subsection is devoted to the convergence of a subsequence of these approximations and to the proof that the limit is a  renormalized solution. 
Then in Subsection \ref{subsec-ren-2}  we prove the strong convergence of  the approximated vorticities
 $ \omega_{\nu}$  in $C_{loc}(\bbR^+; L^p(\calF)) $, not only in $C_{loc}(\bbR^+; L^p(\calF) - w) $, 
 which  concludes the proof of Theorem \ref{the:str}.
The proof makes uses of the transport equation satisfied by the vorticity, where the velocity vector field is associated with the vorticity by   \eqref{vort-ss2-d} and \eqref{vort-ss2-e}.
With a few adaptations, it is also possible to deal with the case of a transport equation when the velocity field $v$ is given rather than associated with the vorticity, this is explained in Subsection \ref{just-transp}.

%%%%%%%%%%%%%%%%%%%
\subsection{Convergence of the approximations to a renormalized solution}
\label{sec-pr-re1}

We start as in the proof of Theorem  \ref{def:wea:sol:ss}: 
for $\nu $ in $(0,1)$,  we consider $ \omega^{in}_{\nu} $, $ \omega_{\nu}^{+} $ and $ g_{\nu} $, as in 
Lemma \ref{reg-data}, and $ \omega_{\nu} $ the corresponding global unique smooth solution of the system \eqref{equ:app:ss} as in Lemma \ref{visc-ex}. 
% These solutions satisfy \eqref{weak:for:SS:MB} for any $ \varphi \in C^{\infty}_{c}(\bbR^{+}\times \overline{\calF}) $. 
By applying Proposition \ref{visc-appr}, in the case where the family $(v_{\nu})_{\nu \in (0,1)}$ is related to the family of vorticity $(\omega_{\nu})_{\nu \in (0,1)}$ by \eqref{u:dec:ss},  we obtain that, up to a subsequence, 
 $( \omega_{\nu} )_\nu$ and $(g_{\nu}^{1/p} \omega_{\nu})_\nu$ satisfy \eqref{cv-fc}, and $(v_{\nu})_{\nu \in (0,1)}$ satisfies 
  \eqref{hyp-vnu}.  Moreover the limit vorticity $\omega$ and the limit velocity $v$  satisfy \eqref{u:dec:ss} and \eqref{vort-ss2-e}.
However we will not try to pass to the limit in the renormalized formulation \eqref{wf:equ:ss:ren}  of the evolution equation but rather proceed by duality, following the strategy used in  \cite{CS} to prove the corresponding result in the case without source nor sink. 
We will proceed in $2$ steps.

  \paragraph{Step 1. Reduction to a duality formula.}
By Proposition \ref{exi-pq}  there is a renormalized solution $ (\bar{\omega}, \bar{\omega}^{-}) $ to the transport equation associated with the vector field $ v $ and the data $ \omega^{in} $ and $ \omega^{+} $
and  to prove Theorem  \ref{exi:Lp:ss} it is sufficient to establish that 
\begin{equation}
  \label{esta}
 (\omega, \omega^{-} ) =  (\bar{\omega}, \bar{\omega}^{-})  .
\end{equation}
To prove  \eqref{esta} it is sufficient to prove that for any $ T > 0 $, for any smooth functions  $ \phi_- $ and $ \phi_{T} $, 
\begin{equation}\label{densi}
  \int_{\calF} ( \omega(T,.) -\bar{\omega}(T,.) )\phi_{T}
    + \int_0^{T} \int_{\partial \calF^{-}} g (\omega^{-} - \bar{\omega}^{-}) \phi_- = 0,
\end{equation}
which we are going to prove thanks to the duality formula. Indeed by Proposition \ref{exi-pq}  the renormalized solution 
 $ (\bar{\omega}, \bar{\omega}^{-}) $ of the transport equation  associated with the vector field $ v $ 
 satisfies 
\begin{equation}  \label{dua2}
\int_{\calF} \bar{\omega}(T,.) \phi_{T}    + \int_0^{T} \int_{\partial \calF^{-}} g \bar{\omega}^{-} \phi_- = \int_{\calF} \omega^{in} \phi(0,.)  - \int_{0}^{T} \int_{\partial \calF^{+}} g \omega^{+} \phi^{+} ,
\end{equation}
where $( \phi,\phi_+)$ is the unique 
 renormalized solution of  
\begin{subequations}  \label{back:trans:equ}
\begin{align}
-\partial_t \phi -v \cdot \nabla \phi = & \, 0 \quad && \text{ in } [0,T] \times \calF, \\
 \phi = & \, \phi_-  \quad && \text{ on } [0,T] \times \partial \calF^{-},  \\
\phi(T,\cdot) = & \, \phi_{T} . && 
\end{align}
\end{subequations}

 It is therefore sufficient to prove that  $ ({\omega}, {\omega}^{-}) $ satisfies the same equation, that is
\begin{equation} \label{dua1bis}
  \int_{\calF} \omega(T,.) \phi_{T} +  \int_0^{T} \int_{\partial \calF^{-}} g \omega^{-} \phi_- = \int_{\calF} \omega^{in} \phi(0,.)  - \int_{0}^{T} \int_{\partial \calF^{+}} g \omega^{+} \phi^{+} ,
\end{equation}
since the difference of   \eqref{dua1bis} and of \eqref{dua2}  leads to \eqref{densi}. 

 \paragraph{Step 2. Proof of the duality formula.}
 To prove \eqref{dua1bis}  we are going to establish first a similar duality formula for the viscous approximations, and then we will pass to the limit as $\nu$ converges to $0$. 
By using  Lemma \ref{visc-ex}, in the second case with $v_{\nu}$ as above, 
there exist a smooth solution $ \phi_{\nu}  $  
to the backward viscous transport equation
\begin{subequations} \label{vis:back:trans:equ}
\begin{align}
-\partial_t \phi_{\nu} -v_{\nu}\cdot \nabla \phi_{\nu} - \nu \Delta \phi_{\nu} = & \, 0\quad && \text{ in } [0,T] \times \calF,  \\
\partial_n \phi_{\nu}  = & \, -( \phi_{\nu} - \phi_-) g_{\nu} \mathds{1}_{\partial \calF^{-}} \quad && \text{ on } [0,T] \times \partial \calF,    \\
\phi_{\nu}(T,\cdot) = & \, \phi_{T}. &&  
\end{align}
\end{subequations}
Moreover, by Proposition \ref{visc-appr}, up a subsequence, 
 the functions $ \phi_{\nu} $ satisfy 
\begin{equation} \label{cv-test}
    \phi_{\nu} \longrightarrow \bar\phi \quad \text{ in } C_{w}([0,T];L^{p^{*}}(\calF))
     \quad \text{ and } \quad 
     g_{\nu}^{1/p^*} \phi_{\nu} \cv g^{1/p^{*}} \bar\phi^{+} \text{ in } L^{p^{*}}((0,T)\times \partial \mathcal{F}^{+})
\end{equation}
where $ p^{*} $ is the dual exponent of $ p $.
Thanks to   \eqref{hyp-vnu} and \eqref{vis:back:trans:equ}, by passing to the limit in \eqref{cv-test}, we obtain that 
$(\bar\phi, \bar\phi^{+})$ satisfies 
 \eqref{back:trans:equ} in the distributional sense. 
 By  Proposition \ref{exi-pq}, it is also a renormalized solution of  \eqref{back:trans:equ}, and by uniqueness, 
 $(\bar\phi, \bar\phi^{+}) = (\phi , \phi_+)$. 
 Now, for $\nu$ in $(0,1)$, since  $ \omega_{\nu} $ is a smooth solution of the system \eqref{equ:app:ss}, it also satisfies the weak formulation 
  \eqref{weak:for:SS:MB}. Considering in particular the test function 
   $  \varphi = \phi_{\nu} $  
  and by an integration by parts of  the term containing $\nabla$ we deduce that 
for any $ \nu >0$,  
\begin{gather}
\label{amel}
- \int_{\calF} \omega_{\nu} (T,.) \phi_{\nu} (T,.) dx
 + \int_{\calF} \omega^{in}_{\nu} \phi_{\nu}(0,.) dx +
\int_0^{T}\int_{\calF} \omega_{\nu} ( \partial_t \phi_{\nu}  +  v_{\nu} \cdot \nabla  \phi_{\nu} +   \nu \Delta   \phi_{\nu} )
\\ \nonumber \quad \quad  = \int_0^{T}\int_{\partial \calF^{+}} g_{\nu} \omega^+_{\nu} \phi_{\nu} + \int_0^{T}\int_{\partial \calF^{-}} g_{\nu} \omega_{\nu}  \phi_{\nu}
+   \nu \int_{\bbR^{+}}\int_{\partial \calF}  \omega_{\nu} \partial_n \phi_{\nu}  .
\end{gather}
Let us mention that the time integrals in  \eqref{weak:for:SS:MB} can be converted into the integrals on $[0,T]$ as above by a standard approximation process by smooth functions of the truncation $ \mathds{1}_{(0,T)} $. 
Using \eqref{vis:back:trans:equ}
 we deduce from \eqref{amel} that for any $ \nu >0$,  
\begin{gather}
\label{amel2}
 \int_{\calF} \omega_{\nu} (T,.) \phi_{T} dx + 
 \int_0^{T}\int_{\partial \calF^{-}} g_{\nu} \omega_{\nu} \phi_-
 =  \int_{\calF} \omega^{in}_{\nu} \phi_{\nu}(0,.) dx 
- \int_0^{T}\int_{\partial \calF^{+}} g_{\nu} \omega^+_{\nu} \phi_{\nu}   .
\end{gather}
Using \eqref{cv-data}, \eqref{cv-fc} and   \eqref{cv-test} we deduce \eqref{dua1bis} and therefore conclude the proof of Theorem \ref{exi:Lp:ss}.

\begin{Remark}
\label{lavori}
In the previous proof we use that $ \omega_{\nu} $ converges to $ \omega $ in $ C^0_w(0,T;L^p(\calF))$ to identify $ \omega $ with $ \bar{\omega} $. Let us note that following the same strategy, it is possible to show the same result in the case where the convergence holds weakly-star in $ L^{\infty}(0,T;L^p(\calF)) $. To do that is enough to consider an inverse flow $ \phi_{\nu} $ which is zero at final time and with a source term. More precisely $ \phi_{\nu} $ satisfies  
\begin{align*}
-\partial_t \phi_{\nu} -v_{\nu}\cdot \nabla \phi_{\nu} - \nu \Delta \phi_{\nu} = & \, \chi \quad && \text{ in } [0,T] \times \calF,  \\
\partial_n \phi_{\nu}  = & \, -( \phi_{\nu} - \phi_-) g_{\nu} \mathds{1}_{\partial \calF^{-}} \quad && \text{ on } [0,T] \times \partial \calF,    \\
\phi_{\nu}(T,\cdot) = & \, 0. &&  
\end{align*}
where $ \chi \in C^{\infty}_{c}((0,T)\times \calF ) $. Using the properties of $ \phi_{\nu} $, we deduce

\begin{equation*}
\int_0^{T}\int_{\calF} (\omega - \bar{\omega})\chi + \int_0^{T} \int_{\partial \calF^{-}} g (\omega^{-} - \bar{\omega}^{-}) \phi_- = 0,
\end{equation*}
which is the analogous of \eqref{densi}. This approach will be used in the proof of Theorem \ref{Theo:2023}. Finally under the hypothesis that 
\begin{equation*}
\omega_{\nu} \cvwstar \omega \text{ in } L^{\infty}(0,T;L^p(\calF)) \quad \text{ and } \quad \omega_{\nu}(T,.) \cv \omega_T \text{ in } L^p(\calF),
\end{equation*}
then by taking an inverse flow $ \phi_{\nu} $ with source term and non-zero initial data, we deduce
\begin{equation}
\label{lavori:1}
 \int_0^{T}\int_{\calF} (\omega - \bar{\omega})\chi +
  \int_{\calF} ( \omega(T,.) -\bar{\omega}(T,.) )\phi_{T}
    + \int_0^{T} \int_{\partial \calF^{-}} g (\omega^{-} - \bar{\omega}^{-}) \phi_- = 0,
\end{equation}
in particular also $ \omega_T = \omega(T,.) = \bar{\omega}(T,.) $.

\end{Remark}

%%%%%%%%%%%%%%%%%%%
\subsection{Strong convergence of the vorticity}
\label{subsec-ren-2}

Above, we have proven that the sequence of approximate solutions $ \omega_{\nu} $  converges to $\omega $,  as the parameter $\nu$ goes to $0$, in $ C_{loc}(\bbR^+;L^{p}(\calF)-w) $ and $ \omega_{\nu}g_{\nu}^{1/p} \cv \omega^{-} g^{1/p} $ in $ L^{p}_{loc}(\bbR^+ \times \partial \calF^{-} )$ where $ (\omega, \omega^{-}) $ is the unique renormalized solution to the Euler system in vorticity form.
Indeed we can prove that  the convergences of $ \omega_{\nu} $ and of $ \omega_{\nu} g_{\nu}^{1/p} $ are  strong respectively in $ C_{loc}^{0}(\R_+;L^p(\calF)) $ and $ L_{loc}^{p}(\bbR^+\times \partial \calF^{-})$, in the case where  $ p $ is in $(1,\infty) $. 

\begin{Theorem}
\label{the:str}
Let  $ p $ in $(1,\infty) $. 
  As  $\nu$ goes to $0$, the sequence 
 $ \omega_{\nu}$ converges to  $\omega $ in $ C^{0}_{loc}(\bbR^+; L^p(\calF)) $
 and the sequence $ \omega_{\nu}g_{\nu}^{1/p} $  converges to  $\omega^{-} g^{1/p} $ in $ L^p_{loc}(\bbR^{+}\times \partial \calF^-)$.
\end{Theorem} 

\begin{proof}[Proof of Theorem \ref{the:str}]

The proof is divided into seven steps.

\paragraph{Step 1.} Let  $ G(x) := x^2/{\sqrt{x^2+1}} $. 
The function $G$ is strictly convex function. Moreover, since $ G(x) \leq |x | $, we deduce from the bounds on the sequences  $ \omega_{\nu} $ and $ \omega_{\nu}g_{\nu}^{1/p} $
that the sequences $ G(\omega_{\nu}) $ and $ G(\omega_{\nu})g_{\nu}^{1/p} $ are uniformly bounded respectively in $ L^{\infty}_{loc}(\bbR^+; L^{p}(\mathcal{F})) $ and $ L_{loc}^{p}(\bbR^+\times \partial \mathcal{F}^-)$. 
As a consequence, up to  subsequences, for any $T>0$, 
\begin{equation}\label{cvG1}
G(\omega_{\nu}) \cvwstar \mathfrak{G} \quad \text{ in } L^{\infty}(0,T;L^p(\calF)) \quad \text{ and } \quad G(\omega_{\nu}) \cv \mathfrak{g} \quad \text{ in } L^p(0,T;L^{p}(\partial \calF)). 
\end{equation}

\paragraph{Step 2.} Recall that $ G(\omega_{\nu}) $ satisfies the system:
\begin{align}
\label{sys:for:G:nu}
\partial_t G(\omega_{\nu}) + v_{\nu}\cdot \nabla G(\omega_{\nu})  =  & \,  \nu \Delta G(\omega_{\nu} ) -  \nu G''(\omega_{\nu} ) \vert  \nabla \omega_{\nu} \vert^2
 \quad && \text{ in  } \bbR^{+} \times \calF,  \\ \label{sys:for:G:nub}
\nu \partial_n G(\omega_{\nu}) = & \, (\omega_{\nu}-\omega_{\nu}^{+})G'(\omega_{\nu})g_{\nu}\mathds{1}_{\calF^+}  \quad && \text{ in } \bbR^{+}\times \partial \calF_+. 
\end{align}   
By some integration by parts,   we deduce that 
\begin{align}
\label{ene:equ:ss:mm:ff}
\int_{\calF} & G(\omega_{\nu}(t,\cdot)) + \int_0^{t} \int_{\partial \calF^{-}} g_{\nu} G(\omega_{\nu})
+ \nu \int_0^t \int_{\calF} G''(\omega_{\nu} ) \vert  \nabla \omega_{\nu} \vert^2  \\ & \, 
+ \int_0^t\int_{\partial \calF^{+}} (G(\omega_{\nu}^+) - G(\omega_{\nu})  + (\omega_{\nu}-\omega_{\nu}^{+})G'(\omega_{\nu})) (-g_{\nu})
=  \int_{\calF} G(\omega_{\nu}^{in}) + \int_0^{t} \int_{\partial \calF^{+}} g_{\nu} G(\omega_{\nu}^{+}) \nonumber .
\end{align}   
Observe that, 
since $G$ is convex and $g_{\nu} $ is negative on $\calF^{+}$, the two last terms in the left hand side are nonnegative. 
Using on the one hand, the weak convergence of the sequences $ \omega_{\nu}$ and $ \omega_{\nu}g_{\nu}^{1/p} $ 
and the weakly lower semicontinuity of the functionals 
$$ \mathcal{G}(f) := \int_{\calF}G(f) \quad \text{and} \quad \mathcal{G}_b(f_1,f_2) := \int_0^t \int_{\partial \calF^{+}}G(f_1)|f_2| ,$$
for the left hand side, and on the other hand the strong convergence of the data to handle the right hand side, 
we deduce that 
\begin{align} \nonumber
\int_{\calF} & G(\omega(t,\cdot)) + \int_0^{t} \int_{\partial \calF^{-}} g G(\omega^-)
+ \liminf \nu \int_0^t \int_{\calF} G''(\omega_{\nu} ) \vert  \nabla \omega_{\nu} \vert^2  \\ \nonumber & \, 
+ \liminf\left( \int_0^t\int_{\partial \calF^{+}} (G(\omega_{\nu}^+) - G(\omega_{\nu})  + (\omega_{\nu}-\omega_{\nu}^{+})G'(\omega_{\nu}))(-g_{\nu})\right)
 \\ \label{acomb}
 & \,  \leq  \int_{\calF} G(\omega^{in}) + \int_0^{t} \int_{\partial \calF^{+}} g G(\omega^{+}).
\end{align} 
Moreover  $ (\omega, \omega^{-}) $ being a renormalized solution to the transport equation associated with the velocity $v$ and the data $\omega^{in}$ and $\omega^{+}$, 
it holds in particular that 
\begin{align}
\label{ene:equ:ss:mm}
\int_{\calF} & G(\omega(t,\cdot)) + \int_0^{t} \int_{\partial \calF^{-}} g G(\omega^-) =  \int_{\calF} G(\omega^{in}) + \int_0^{t} \int_{\partial \calF^{+}} g G(\omega^{+}).
\end{align} 
Combining  \eqref{acomb} and \eqref{ene:equ:ss:mm} we arrive at 
\begin{align} \nonumber
 \liminf \nu \int_0^t \int_{\calF} G''(\omega_{\nu} ) \vert  \nabla \omega_{\nu} \vert^2  
+ \liminf\left( \int_0^t\int_{\partial \calF^{+}} (G(\omega_{\nu}^+) - G(\omega_{\nu})  + (\omega_{\nu}-\omega_{\nu}^{+})G'(\omega_{\nu}))(-g_{\nu})\right)
  \leq  0.
\end{align} 
Therefore 
\begin{gather}\label{cvG2}
\Big (\nu G''(\omega_{\nu} ) \vert  \nabla \omega_{\nu} \vert^2 , (G(\omega_{\nu}^+) - G(\omega_{\nu})  + (\omega_{\nu}-\omega_{\nu}
^{+})G'(\omega_{\nu}))g_{\nu} \Big) \to 0 \, \nonumber \\  \text{ in } L^1_{loc}(\bbR^{+};L^{1}( \calF)) \times L^{1}_{loc}(\bbR^{+} \times \partial \calF^+). 
\end{gather} 

\paragraph{Step 3.}
 Let  $q$  such that   $p$ and $q$  are conjugated. 
Let $\Psi $ in $L^q ( [0,T] \times \partial \calF^{-} ;g \, dz \, ds)$ and 
$\chi $ in $L^1  ( [0,T]  ;L^q( \calF))$. 
Let $(\phi, \phi^+) $  the  $L^q$ renormalized solution of the backward transport equation:
\begin{align*}
-\partial_t \phi -v \cdot \nabla \phi = & \, \chi \quad && \text{ in } [0,T] \times \calF, \nonumber \\
 \phi = & \, \Psi  \quad && \text{ on } [0,T] \times \partial \calF^{-},  \label{rrr} \\
\phi(T,x) = & \,  0 .
&&  \nonumber
\end{align*}
Then, proceeding as in the proof of Theorem \ref{exi:Lp:ss},   
 with the help of the convergences  \eqref{cvG1} and \eqref{cvG1}, we obtain:  
\begin{equation} \label{duaG1}
    \int_0^{T}\int_{\calF}\mathfrak{G} \chi + \int_0^{T} \int_{\partial \calF^{-}} g \mathfrak{g} \Psi = \int_{\calF} G(\omega^{in}) \phi(0,.) 
    - \int_{0}^{T} \int_{\partial \calF^{+}} g G(\omega^{+}) \phi^{+}.
\end{equation}
On the other hand, since  $ (\omega, \omega^{-}) $ is a renormalized solution to the transport equation associated with the velocity $v$ and the data $\omega^{in}$ and $\omega^{+}$, it satisfies  the duality formula:
\begin{equation} \label{duaG2}
    \int_0^{T}\int_{\calF} G(\omega) \chi + \int_0^{T} \int_{\partial \calF^{-}} g G(\omega^{-}) \Psi = \int_{\calF} G(\omega^{in}) \phi(0,.) 
    - \int_{0}^{T} \int_{\partial \calF^{+}} g G(\omega^{+}) \phi^{+}.
\end{equation}
By combining \eqref{duaG1} and \eqref{duaG1}, we 
deduce that  $ G(\omega) = \mathfrak{G} $ and $ G(\omega^{-}) = \mathfrak{g} $.

\paragraph{Step 4.} 
We recall  the following result which is proved in the first step of the proof of Lemma 3.34 of \cite{NovS}.
\begin{Lemma}
\label{lem:ref}
 Let $ \mathcal{O} \subset \bbR^2 $ measurable, let $ f_{n} $, $ f $ a sequence of $ L^1(\mathcal{O}) $ functions and $ G $ a strictly convex function. If $ f_{n} \cv f $ and $ G(f_n) \cv G(f) $ in $ L^1(\mathcal{O}) $. Then,  up to a subsequence,  $ f_{n} $ converges  almost everywhere to $ f $ pointwise.  
\end{Lemma} 
To handle the convergence on the boundary we will also use the following corollary, which is proved in an appendix. 
\begin{Corollary}
\label{cor:ref}
Let $  p $ in $(1,+\infty)$,  let $ f_{\nu} $, $ f $ a sequence  of $ L^p((0,T) \times \partial \calF^{-};g_{\nu} \, dtds) $ functions and $ G $ a strictly convex function from $\R$ to $\R$ with bounded derivative. If $ f_{\nu} g_{\nu} \cv f g $ and $ G(f_{\nu}) g_{\nu} \cv G(f)g $ in $ L^p((0,T) \times \partial \calF^{-}) $. Then,  up to subsequence, $ f_{\nu} $ converges almost everywhere to $ f $. 
\end{Corollary}
From the previous steps, Lemma \ref{lem:ref} and Corollary \ref{cor:ref}, we deduce the almost everywhere convergences of the sequence $ \omega_{\nu} $ to $ \omega $ in $ (0,t) \times \calF $ and of the sequence  $ \omega_{\nu} $ to $ \omega^{-} $ in $ [0,T]\times \partial \calF^- $. 

\paragraph*{Step 5.} 
From De la Vall\'ee Poussin's lemma and Corollary \ref{Cor:DLVPoussin} together with the estimates from Step 1 of Proposition~\ref{visc-appr} we deduce that 
 the sequences $ |\omega_{\nu}|^p $ and $ g_{\nu}|\omega_{\nu}|^p|_{\partial \calF^{-}} $ are uniformly integrable.
Then, by the previous step and Vitali's Lemma (see for example Theorem 1.18 of \cite{NovS}), we have that 
\begin{equation*}
\omega_{\nu} \to \omega \text{ in } L^q(0,T;L^{p}(\calF)) \text{ for any  } q \in [1,+\infty) \quad \text{ and } \quad \omega_{\nu}g_{\nu}^{1/p}  \to  \omega^{-} g^{1/p} \quad \text{ in } L^p((0,T)\times \partial \calF^-).
\end{equation*}

\paragraph{Step 6.} The convergence $ \omega_{\nu} $ to $ \omega $ in $ L^q(0,T;L^p(\calF) ) $, implies that that $ \omega_{\nu}(t,.) $ converges to $ \omega(t,.) $ in $ L^p(\calF) $ for almost any time. Let now prove by contradiction that the convergence holds for any time. Suppose that there exists $ s \in (0,T] $ such that the convergence does not hold. In particular there is $\eps>0$ and there exists a subsequence of $( \omega_{\nu})$, which we do not relabel, such that $ \| \omega_{\nu}(s,.)- \omega(s,.)\|_{L^p(\calF)} \geq \eps > 0  $. From the \textit{a priori} estimates \eqref{est:con:Gun}, with $ t = s $, we deduce that passing to subsequences $ G(\omega_{\nu}(s,.)) \cv \mathfrak{G}_s $ in $ L^{p}(\calF) $. Moreover proceeding as mentioned in Remark \ref{lavori}, we deduce the identification formula \eqref{lavori:1}. In particular we identify $ \mathfrak{G}_s $ with $ G(\omega(s,.)) $. From Lemma \ref{lem:ref} and the weak convergence of $ \omega_{\nu}(s,.) $ to $ \omega(s,.) $, we obtain that up to subsequence $ \omega_{\nu}(s,.) $ converges strongly to $ \omega(s,.) $ in $ L^p(\calF)$ which is a contradiction. We have shown that $ \omega_{\nu}(t,.) $ converges to $ \omega(t,.) $ for any $ t \in [0,T]$.

\paragraph{Step 7.} We conclude by showing that the convergence holds uniformly in time. Suppose by contradiction that the convergence is not uniform. Then there exists $ \delta > 0 $, there exists a subsequence of $( \omega_{\nu})$, which we do not relabel,  and a sequence of times $ t_{n} $ such that $ \| \omega_{\nu}(t_{n}) - \omega(t_n) \|_{L^p(\calF)} \geq 2\delta $. The interval $ [0,T] $ is compact. Passing to a subsequence we can assume that $t_n \to t $. Using the continuity of $ \omega $ from $ n $ large enough  we have that $ \| \omega_{\nu}(t_{n}) - \omega(t) \|_{L^p(\calF)} \geq \delta $. 
On the other hand it follows from the $ C_{w}(L^p) $ convergence of $ \omega_{\nu} $ to $ \omega $ that  $ \omega_{\nu}(t_{n})  $ converges to $ \omega(t) $ weakly in $ L^p(\calF) $ and   combining \eqref{ene:equ:ss:mm:ff} and \eqref{ene:equ:ss:mm} we obtain that $\| \omega_{	\nu}(t_{n}) \|_{L^{p}(\calF)} $ converges to $\| \omega(t) \|_{L^{p}(\calF)} $. 
Since for $ 1 < p < \infty $, the space  $ L^p $ is  uniformly convex, and we deduce that 
 $ \omega_{\nu}(t_{n})  $ converges to $ \omega(t) $ in $ L^p(\calF) $, which is the desired contradiction. 
This concludes the proof of Theorem \ref{the:str}.

\end{proof}

%%%%%%%%%%%%%%%%%%%%%%%%%%%%
\subsection{A note on the vanishing viscosity solution to the transport equation}
\label{just-transp}

Since the proof above uses the transport equation satisfied by the vorticity, one may wonder if similar results are true for the transport equation when the velocity field $v$ is given rather than associated with the vorticity. 
Indeed in the case where the fluid occupies the whole space $ \bbR^d $, 
 it is known, see \cite[Theorem IV.1]{DL} and more recently in \cite{NSW}, that the vanishing viscosity approximations converge
   in $ C([0,T];L^p) $ to the renormalized solution to the transport equation.
These results can be extended to the present setting, where sources and sinks are present.  However, compared to the proof above, some adaptations are needed.  In particular,  in Subsection \ref{subsec-ren-2}, we took advantage of the convergence of the vorticity in $ C([0,T];L^p - w) $, which it is not clear in the case where  the velocity field is not related to the vorticity through a div-curl system. 
For sake of completeness, we state and sketch the proof of the corresponding result.

\begin{Theorem}
\label{Theo:2023}
Let $p $  in $[1,+\infty]$ and $ q \in (1,+\infty) $. 
Let $ v $ in $ L^{1}_{loc}(\bbR^{+}; W^{1,q}(\calF)) $ a \ref{SSC:v} vector field. 
Let  $(\omega^{in} , \omega^{+}  )$ a (\ref{CIV})  in $L^p$. 
Let some families $ \omega^{in}_{\nu} $, $ \omega_{\nu}^{+} $, $ v_{\nu} $ as in 
Lemma \ref{reg-data}. 
Let  $ \omega_{\nu} $ the corresponding global unique smooth solution of the system \eqref{equ:app:ss} associated with the data $ \omega^{in}_{\nu} $, $ \omega_{\nu}^{+} $ and $ v_{\nu}$ as in Lemma \ref{visc-ex}. 
Then there exists a subsequence of $ \omega_{\nu} $ which we still denote  $ \omega_{\nu} $ and $ v_{\nu }$  and which satisfy for $ p < \infty $
\begin{gather} \label{cv-fc-kk}
    \omega_{\nu} \cvwstar \omega \quad \text{ in } L^{\infty}([0,T];L^{p}(\calF)),
    \quad 
    g_{\nu}^{1/p} \omega_{\nu} \cv g^{1/p} \omega^{-}  \text{ in } L^{p}((0,T)\times \partial \mathcal{F}^{-}), \\
\text{ and } \quad  \sqrt{\nu}  \nabla \omega_{\nu}  \longrightarrow 0 \quad \text{ in } L^{2}([0,T];L^2(\calF)). \nonumber
\end{gather}
Moreover  $ (\omega, \omega^{-}) $ is the unique renomalized solution to the transport associated with the velocity field $ v $ and the data $ (\omega^{in} ,  \omega^+) $.  
Finally, if $p, q \in (1,\infty)$, then the two first convergences in  \eqref{cv-fc-kk} can be improved into the following strong convergences: 
 $ \omega_{\nu} $ converges to $  \omega $ in $ C^{0}_{loc}(\bbR^+; L^p(\calF)) $ and 
  $ \omega_{\nu}g_{\nu}^{1/p} $ converges to $  \omega^{-} g^{1/p} $ in $ L^p_{loc}(\bbR^{+}\times \partial \calF^-)$.
\end{Theorem}
 
\begin{proof}
The proof of  \eqref{cv-fc-kk}
can be performed  as in the proof of Proposition \ref{visc-appr}.
Moreover, by using the duality method as in the proof of Theorem \ref{exi:Lp:ss}, and taking account Remark \ref{lavori} we
prove that $ (\omega, \omega^{-}) $ is the unique renomalized solution to the transport associated with the velocity field $ v $ and the data $ (\omega^{in} ,  \omega^+) $. 
  Finally it is possible to show the strong convergence of the vorticity as follows. 
For $ q \geq 2 $, we deduce that $ \omega_{\nu} \to \omega $ in $ C^0_{w}([0,T];L^{p}(\calF)) $ by using the equations so that one may proceed as in the case of the Euler system, see Subsection \ref{subsec-ren-2}.
For $q<2$, we can proceed as in \cite{NSW}, by considering a sequence  $ (\omega^{in}_{l}, \omega^{-}_{l}) \in L^{q'} $  converging  to $ (\omega^{in}, \omega^{-}) $ in $ L^p $. From Lemma \ref{reg-data} there exist sequences $ (\omega_{l,\nu}^{in}, \omega_{l,\nu}^+) $ of compatible regular data. Since $ q < 2 $, the estimates \eqref{conv-visc-appr} hold true and we can furthermore impose the extra condition
$$ \sup_{l}\left(\|\omega_{l,\nu}^{in}-\omega_{l}^{in}\|_{L^p(\calF)} + \|g_{\nu}^{1/p}\omega_{l,\nu}^+-g^{1/p}\omega_{l}^+\|_{L^{p}((0,T)\times \partial \calF^-)}\right) \longrightarrow 0 \quad \text{ as } \nu \longrightarrow 0. $$  
Then, by the triangle inequality, 
\begin{align*}
\|\omega_{\nu}(t,.) & \, -\omega(t,.)\|_{L^p(\calF)} + \|g_{\nu}^{1/p}\omega_{\nu}^+-g^{1/p}\omega^+\|_{L^{p}((0,t)\times \partial \calF^-)} \leq \\ & \,  \|\omega_{\nu}(t,.)-\omega_{l,\nu}(t,.)\|_{L^p(\calF)} + \|g_{\nu}^{1/p}\omega_{\nu}^+-g_{\nu}^{1/p}\omega_{l,\nu}^+\|_{L^{p}((0,t)\times \partial \calF^-)} \\ & \,  + \|\omega_{l,\nu}(t,.)-\omega_{l}(t,.)\|_{L^p(\calF)} + \|g_{\nu}^{1/p}\omega_{l,\nu}^+-g^{1/p}\omega_{l}^+\|_{L^{p}((0,t)\times \partial \calF^-)}   \\ & \,\|\omega_{l}(t,.)-\omega(t,.)\|_{L^p(\calF)} + \|g^{1/p}\omega_{l}^+-g^{1/p}\omega^+\|_{L^{p}((0,t)\times \partial \calF^-)}.
\end{align*}

Note that for the first and the last line of the right hand side the norm can be bounded by a constant times the size of the initial data, and therefore converge to zero. 
Moreover  the middle term converges to zero for any fixed $ l $, thanks to  the previous step since $ q' > p $.  

\end{proof}

%%%%%%%%%%%%%%%%%%%%%%%%5
\section{Proof of Theorem \ref{exi:L1:ss} on  symmetrized solutions}
 \label{sec-proof-sym}

 We start as in the proof of Theorem  \ref{def:wea:sol:ss} and of Theorem  \ref{exi:Lp:ss}: 
 for $\nu $ in $(0,1)$,  we consider $ \omega^{in}_{\nu} $, $ \omega_{\nu}^{+} $ and $ g_{\nu} $, as in 
Lemma \ref{reg-data}, with $p=1$, and $ \omega_{\nu} $ the corresponding global unique smooth solution of the system \eqref{equ:app:ss} as in Lemma \ref{visc-ex}. 
Lemma \ref{reg-data}, and $ \omega_{\nu} $ the corresponding global unique smooth solution of the system \eqref{equ:app:ss} as in Lemma \ref{visc-ex}. 
 These solutions satisfy  the weak formulation \eqref{wf:1:equ:ss-nu}
for any $ \varphi \in C^{\infty}_{c}(\bbR^{+}\times \overline{\calF}) $,
where the $\calC_{i,\nu}$ are given by \eqref{formuCi-nu}.
By applying Proposition \ref{visc-appr}, in the case where the family $(v_{\nu})_{\nu \in (0,1)}$ is related to the family of vorticity $(\omega_{\nu})_{\nu \in (0,1)}$ by \eqref{u:dec:ss},  we obtain that, up to a subsequence, 
\eqref{cv-fc} holds true  with $p=1$ and the limits  $(\omega , \omega^{-}  )$ of the subsequence, respectively in the domain and on the outgoing part of the boundary, 
   satisfy     \eqref{apq} with $q=1$. 
Let us now examine how to pass to the limit \eqref{equ:app:ss} and  \eqref{formuCi-nu}  as $\nu$ goes to $0$.  
\begin{itemize}
\item Thanks to \eqref{cv-data},  
\begin{align}
\label{T1}
\int_{\calF} \omega_{\nu}^{in}\varphi(0,.) dx \rightarrow \int_{\calF} \omega_{\nu}^{in}\varphi(0,.) dx 
  \text{ and }
  \int_{\bbR^+}\int_{\partial \calF^{+}} g_{\nu} \omega^+_{\nu} \varphi ds dt  \rightarrow 
   \int_{\bbR^+}\int_{\partial \calF^{+}} g \omega^+ \varphi ds dt .
\end{align}
In particular, from the last convergence we deduce that 
\begin{align}
\label{T1bis}
\calC_{i,\nu} \rightarrow \calC_{i} \, \text{ in  } \, C([0,T]) , 
\, \text{ with }  \calC_{i}(t) :=  \calC_i^{in} - \int_{0}^{t} \int_{\pS^{i}} \omega^+ g \, \text{ for } i \in \calI^+ .
\end{align}
\item Thanks to \eqref{cv-fc}, up to a subsequence, 
\begin{align}
\label{T2}
& \int_{\bbR^{+}}  \int_{\calF} \omega_{\nu} \big( \partial_t \varphi + v_g \cdot \nabla \varphi \big) \, dx dt
  \rightarrow 
   \int_{\bbR^{+}}  \int_{\calF} \omega \big( \partial_t \varphi + v_g \cdot \nabla \varphi \big) \, dx dt , \\ & \label{T2b}
     \int_{\calF} \omega_{\nu} X_i \cdot \nabla \varphi  \, dx   \rightarrow 
        \int_{\calF} \omega X_i \cdot \nabla \varphi  \, dx , \quad  \text{ in } C([0,T])   \, \text{ for } i \in \calI ,
          \end{align}
  and 
\begin{align} \label{T2c}
  \int_{\bbR^+}\int_{\partial \calF^{-}} g^-_{\nu} \omega_{\nu} \varphi ds dt  \rightarrow 
   \int_{\bbR^+}\int_{\partial \calF^{-}} g \omega^- \varphi ds dt , 
   \end{align}
   In particular, from the last convergence we deduce that 
\begin{align}
\label{T2bis}
   \calC_{i,\nu} \rightarrow \calC_{i} \, \text{ in  } \, C([0,T]) , 
\, \text{ with } \calC_{i}(t) :=  \calC_i^{in} - \int_{0}^{t} \int_{\pS^{i}} \omega^- g \ \text{ for } i \in \calI^- .
\end{align}
Therefore \eqref{formuCi} is already proved and moreover, 
 from \eqref{T1bis}, \eqref{T2b} and \eqref{T2bis}, we deduce that 
\begin{align}
\label{T4}
 \sum_{i}  \int_{\bbR^+}  \calC_{i,\nu}(t)  \int_{\calF} \omega_{\nu} X_i \cdot \nabla \varphi  \, dx \, dt 
  \rightarrow 
   \sum_{i}  \int_{\bbR^+}  \calC_{i}(t)  \int_{\calF} \omega X_i \cdot \nabla \varphi  \, dx \, dt .
 \end{align}

\item Thanks to \eqref{cv-fc}, the viscous term converges to zero: 
\begin{align}
\label{T5}
 \nu \int_{\bbR^{+}}\int_{\calF} \nabla \omega_{\nu} \cdot \nabla \varphi   \rightarrow  0  .
\end{align}
\item  We are now going to prove that for any test function  $\varphi $  in $ \mathfrak{C}_{0}(\calF) $ (recall the definition above Lemma  \ref{Hbd}),
\begin{align}
\label{T6}
\int_{\bbR^+}\int_{\calF} \int_{\calF} H_{\varphi}(x,y)\omega_{\nu} (t,x)\omega_{\nu}(t,y) \, dx \, dy \, dt  \rightarrow  
 \int_{\bbR^+}\int_{\calF} \int_{\calF} H_{\varphi}(x,y)\omega (t,x)\omega (t,y) \, dx \, dy \, dt  .
\end{align}
To prove that  we consider a $C^\infty$ function $\zeta : [0,+\infty) \rightarrow [0,+\infty)$ such that  $\zeta (x) = 1$ for $x \leq 1$ and  $\zeta (x) = 0$ for $x \geq 2$.
Let $\delta > 0$ and  $\zeta_\delta : \mathcal F \times \mathcal F \rightarrow \R$ such that $\zeta_\delta (x,y) = \zeta (\frac{\vert x-y\vert}{\delta})$. 
The family containing the measures $dt \otimes d\omega_{\nu} \otimes d\omega_{\nu}$, for $\nu$ in $(0,1)$, and  $dt \otimes d\omega \otimes d\omega$ are uniformly integrable in $L^1 ((0,T) \times \calF  \times \calF )$. Therefore for any $\eps >0$ there exists $\delta >0$ such that 
\begin{align}
\label{T6a}
\left\vert \int_{\bbR^+}\int_{\calF} \int_{\calF} \zeta(\frac{\vert x-y\vert}{\delta} ) H_{\varphi}(x,y)\omega_{\nu} (t,x)\omega_{\nu}(t,y) \, dx \, dy \, dt \right\vert
\\ \nonumber + \left\vert
 \int_{\bbR^+}\int_{\calF} \int_{\calF}  \zeta(\frac{\vert x-y\vert}{\delta} )  H_{\varphi}(x,y)\omega (t,x)\omega (t,y) \, dx \, dy \, dt \right\vert \leq \eps  ,
\end{align}
using that the Lebesgue measure of the set of the couples $(x,y)$ in $\calF \times \calF$ such that $\vert x-y\vert \leq 2 \delta$ goes to zero as $\delta$ goes to zero and Lemma  \ref{Hbd} regarding the boundedness of $H_{\varphi}$. 
Morever, for this $\delta$, since the function $(x,y) \mapsto (1- \zeta(\frac{\vert x-y\vert}{\delta} ) H_{\varphi}(x,y)$ is continuous on $(0,T) \times  \calF  \times \calF$ and  the tensor product  $\omega_{\nu} \otimes \omega_{\nu}$ converges to  $\omega \otimes \omega$ in $C_{w}([0,T]; \mathcal M(\calF \times \calF)) $, 
 there exists $\nu>0$ small enough for 
\begin{align}
\label{T6B}
\Bigg\vert \int_{\bbR^+}\int_{\calF} \int_{\calF} (1- \zeta(\frac{\vert x-y\vert}{\delta} )) H_{\varphi}(x,y)\omega_{\nu} (t,x)\omega_{\nu}(t,y) \, dx \, dy \, dt 
\\ \nonumber \quad -
 \int_{\bbR^+}\int_{\calF} \int_{\calF}  (1- \zeta(\frac{\vert x-y\vert}{\delta} ))  H_{\varphi}(x,y)\omega (t,x)\omega (t,y) \, dx \, dy \, dt \Bigg\vert \leq \eps  ,
\end{align}
Gathering \eqref{T6a} and \eqref{T6a} we obtain \eqref{T6a}. 
\end{itemize}

Combining \eqref{wf:1:equ:ss-nu}, \eqref{formuCi-nu},  \eqref{T1},  \eqref{T2},  \eqref{T2b},  \eqref{T2c},  \eqref{T4},  \eqref{T5},  and  \eqref{T6} 
we arrive at \eqref{wf:1:equ:ss}. 
This concludes the proof of Theorem  \ref{exi:L1:ss}.

\appendix

\section{Proof of Corollary \ref{cor:ref}}

In this appendix we prove Corollary \ref{cor:ref} which is a variation Lemma 3.34 of \cite{NovS}.
For the reader's convenience we recall the statement of Corollary \ref{cor:ref}: for $ p$ in $(1, +\infty) $, for any sequence $ f_{\nu} $ in $ L^p((0,T) \times \partial \calF^{-};g_{\nu} \, dtds) $, for any $f$ in $ L^p((0,T) \times \partial \calF^{-};g \, dtds) $  and for any  strictly convex function $ G $ from $\R$ to $\R$  with bounded derivative, if $ f_{\nu} g_{\nu} \cv f g $ and $ G(f_{\nu}) g_{\nu} \cv G(f)g $ in $ L^p((0,T) \times \partial \calF^{-} ; dtds) $,  then,  up to a subsequence, $ f_{\nu} $ converges to $ f $ almost everywhere. 

Let us also recall that $ g_{\nu} $ and $ g $ are positive functions on $ \bbR^{+} \times \partial \calF^- $ such that $ g_{\nu} $ converges to $ g $ in $ L^p ((0,T) \times \partial \calF^{-} ; dtds)  $.

\begin{proof}[Proof of Corollary \ref{cor:ref}.] By the hypothesis that $ g > 0 $ and by  Egoroff's theorem, we obtain that, for any $ \eps > 0 $, there is a measurable set $ E \subset (0,T) \times \partial \calF$ with Lebesgue measure $ \mu(E) < \eps$ and a constant $c_{\eps} > 0 $
such that 
$ g \geq c_{\eps} $ in $ E^c := \big((0,T)\times \partial \calF \big)\setminus E  $,
and such that  $ g_{\nu} $ converges to $ g $ uniformly in $E^c$. 

Since the function $ G $  is strictly convex, for any $x$ in $\R$ there exists a strictly increasing function $ \psi_x $ on $ \bbR^+ $ with  $ \psi_x (0) = 0$,  
such that  for any $y$ in $\R$, 
\begin{equation*}
G(y)-G(x)-G'(x)(y-x) \geq \psi_x(|y-x|).
\end{equation*}
We refer here to the proof of  Lemma 3.34 in \cite{NovS} for more on the construction of such a  function $ \psi_x $. 
We deduce that 
\begin{equation*}
\int_{E^c} G(f_{\nu})g - G(f)g -G'(f)(f_{\nu}-f)g \geq  \int_{E^c} \psi_f(|f_{\nu}-f|) g ,
\end{equation*}
Since $ g $ is positive, it is therefore sufficient to prove that the left-hand side converges to zero to conclude the proof of Corollary \ref{cor:ref}. 
We have
\begin{align*}
\int_{E^c} G(f_{\nu})g - G(f)g -G'(f)(f_{\nu}-f)g  = & \,  \int_{E^c} G(f_{\nu})g_{\nu} - G(f)g -G'(f)(f_{\nu}g_{\nu}-fg) \\ & \, + \int_{E^c} G(f_{\nu})(g-g_{\nu}) -G'(f) f_{\nu} (g-g_{\nu}). 
\end{align*} 
The first line of the right hand side converges to zero from the weak convergences of  $ f_{\nu} g_{\nu} $ to $f g $ and of $ G(f_{\nu}) g_{\nu} $ to $ G(f)g $ in $ L^p((0,T) \times \partial \calF^{-}) $. 
 For the second one we first use the lower bound of $ g $ and 
 the uniform convergence of $ g_{\nu} $ to $ g $  to infer from the weak convergences above that $ f_{\nu}  \cv  f  $ and that $ G(f_{\nu})  \cv G(f) $ in $ L^p((0,T) \times \partial \calF^{-}) $.  This entails that the term on the second line converges to $0$. 
 This concludes the proof.

\end{proof}

\smallskip 
\smallskip 
\smallskip 
\noindent
{\bf Acknowledgements.}

The authors are partially supported by the Agence Nationale de la Recherche, Project IFSMACS, grant ANR-15-CE40-0010, Project SINGFLOWS,
ANR-18-CE40-0027-01, Project BORDS, grant ANR-16-CE40-0027-01, the Conseil R\'egionale d'Aquitaine, grant 2015.1047.CP, the Del Duca Foundation, and the H2020-MSCA-ITN-2017 program, Project ConFlex, Grant ETN-765579.  
M.B. is also supported by the ERCEA under the grant 014 669689-HADE and also by the Basque Government through the BERC 2014-2017 program and by Spanish Ministry of Economy and Competitiveness MINECO: BCAM Severo Ochoa excellence accreditation SEV-2013-0323. 
This work was partly accomplished while F.S. was participating in a program hosted by the Mathematical Sciences Research Institute in Berkeley, California, during the Spring 2021 semester, and supported by the National Science Foundation under Grant No. DMS-1928930.

The authors warmly thank Maria  Kazakova, Gennady Alekseev and Alexander Mamontov for their kind help regarding the russian litterature on the subject.

\end{document}